\documentclass[12pt]{amsart}
\usepackage{txfonts}      %{article} was 12pt latex e
\usepackage{amssymb}
\usepackage{eucal}
\usepackage{amsmath}
\usepackage{amscd}
\usepackage{xcolor}
\usepackage{multicol}
\usepackage[all]{xy}           %xypic macro for latex2.09
\usepackage{graphicx}
\usepackage{color}
\usepackage{colordvi}
\usepackage{xspace}
\usepackage{tikz}
\usepackage{makecell}
\usepackage{appendix}
\usepackage{amsthm}

\usepackage{ifpdf}
\ifpdf
\usepackage[colorlinks,final,backref=page,hyperindex]{hyperref}
\else
\usepackage[colorlinks,final,backref=page,hyperindex,hypertex]{hyperref}
\fi

\usepackage[active]{srcltx} %SRC Specials for DVI Searching

\begin{document}

%%%%%%%%%%%%%%%%%%%%%%%% Statements

\newtheorem{thm}{Theorem}[section]
\newtheorem{lem}[thm]{Lemma}
\newtheorem{cor}[thm]{Corollary}
\newtheorem{pro}[thm]{Proposition}
\theoremstyle{definition}
\newtheorem{defi}[thm]{Definition}
\newtheorem{ex}[thm]{Example}
\newtheorem{rmk}[thm]{Remark}
\newtheorem{pdef}[thm]{Proposition-Definition}
\newtheorem{condition}[thm]{Condition}

\renewcommand{\labelenumi}{{\rm(\alph{enumi})}}
\renewcommand{\theenumi}{\alph{enumi}}

\newcommand {\emptycomment}[1]{} %to remove paragraphs

\newcommand{\nc}{\newcommand}
\newcommand{\delete}[1]{}

\nc{\tred}[1]{\textcolor{red}{#1}}
\nc{\tblue}[1]{\textcolor{blue}{#1}}
\nc{\tgreen}[1]{\textcolor{green}{#1}}
\nc{\tpurple}[1]{\textcolor{purple}{#1}}
\nc{\tgray}[1]{\textcolor{gray}{#1}}
\nc{\torg}[1]{\textcolor{orange}{#1}}
\nc{\tmag}[1]{\textcolor{magenta}}
\nc{\btred}[1]{\textcolor{red}{\bf #1}}
\nc{\btblue}[1]{\textcolor{blue}{\bf #1}}
\nc{\btgreen}[1]{\textcolor{green}{\bf #1}}
\nc{\btpurple}[1]{\textcolor{purple}{\bf #1}}

\nc{\todo}[1]{\tred{To do:} #1}

%\delete{
    \nc{\mlabel}[1]{\label{#1}}  % Use this to suppress names
    \nc{\mcite}[1]{\cite{#1}}  % Use this to suppress names
    \nc{\mref}[1]{\ref{#1}}  % Use this to suppress names
    \nc{\meqref}[1]{\eqref{#1}}  % Use this to suppress names
    \nc{\mbibitem}[1]{\bibitem{#1}} % Use this to show number
%}

\delete{
    \nc{\mlabel}[1]{\label{#1}    { {\small\tgreen{\tt{{\ }(#1)}}}}} % Use the next line to show names
    \nc{\mcite}[1]{\cite{#1}{\small{\tt{{\ }(#1)}}}}  % Use this lines to show names
    \nc{\mref}[1]{\ref{#1}{\small{\tred{\tt{{\ }(#1)}}}}}  % Use this lines to show names
    \nc{\meqref}[1]{\eqref{#1}{{\tt{{\ }(#1)}}}}  % Use this lines to show names
    \nc{\mbibitem}[1]{\bibitem[\bf #1]{#1}} % Use this to show name
}

%    \nc{\mlabel}[1]{  % Use the next two lines to show names
%           { {\small\tgreen{\tt{{\ }(#1)}}}}}

\nc{\cm}[1]{\textcolor{red}{Chengming:#1}}
\nc{\yy}[1]{\textcolor{blue}{Yanyong: #1}}
%\nc{\lit}[2]{\textcolor{blue}{#1}{ \textcolor{purple}{(#2)}}}
%\nc{\lit}[2]{\textcolor{blue}{#1}{}} %use this line instead of the previous one to show only the new changes
\nc{\li}[1]{\textcolor{purple}{#1}}
\nc{\lir}[1]{\textcolor{purple}{Li:#1}}

\nc{\revise}[1]{\textcolor{blue}{#1}}

%%%%%%%% new symbols

\nc{\tforall}{\ \ \text{for all }}
\nc{\hatot}{\,\widehat{\otimes} \,}
\nc{\complete}{completed\xspace}
\nc{\wdhat}[1]{\widehat{#1}}

\nc{\ts}{\mathfrak{p}}
\nc{\mts}{c_{(i)}\ot d_{(j)}}

\nc{\NA}{{\bf NA}}
\nc{\LA}{{\bf Lie}}
\nc{\CLA}{{\bf CLA}}
\nc{\gda}{GD algebra\xspace}
\nc{\gdas}{GD algebras\xspace}
\nc{\gdba}{GD bialgebra\xspace}
\nc{\gdbas}{GD bialgebras\xspace}
\nc{\cybe}{CYBE\xspace}
\nc{\nybe}{NYBE\xspace}
\nc{\ccybe}{CCYBE\xspace}

\nc{\ndend}{pre-Novikov\xspace}
\nc{\calb}{\mathcal{B}}
\nc{\rk}{\mathrm{r}}
\newcommand{\g}{\mathfrak g}
\newcommand{\h}{\mathfrak h}
\newcommand{\pf}{\noindent{$Proof$.}\ }
\newcommand{\frkg}{\mathfrak g}
\newcommand{\frkh}{\mathfrak h}
\newcommand{\Id}{\rm{Id}}
\newcommand{\gl}{\mathfrak {gl}}
\newcommand{\ad}{\mathrm{ad}}
\newcommand{\add}{\frka\frkd}
\newcommand{\frka}{\mathfrak a}
\newcommand{\frkb}{\mathfrak b}
\newcommand{\frkc}{\mathfrak c}
\newcommand{\frkd}{\mathfrak d}
\newcommand {\comment}[1]{{\marginpar{*}\scriptsize\textbf{Comments:} #1}}

\nc{\vspa}{\vspace{-.1cm}}
\nc{\vspb}{\vspace{-.2cm}}
\nc{\vspc}{\vspace{-.3cm}}
\nc{\vspd}{\vspace{-.4cm}}
\nc{\vspe}{\vspace{-.5cm}}

%%%%%%%%%%%%%%%%%%%%%%% old symbols

\nc{\disp}[1]{\displaystyle{#1}}
\nc{\bin}[2]{ (_{\stackrel{\scs{#1}}{\scs{#2}}})}  %binomial coeff
\nc{\binc}[2]{ \left (\!\! \begin{array}{c} \scs{#1}\\
    \scs{#2} \end{array}\!\! \right )}  %binomial coeff
\nc{\bincc}[2]{  \left ( {\scs{#1} \atop
    \vspace{-.5cm}\scs{#2}} \right )}  %binomial coeff
\nc{\ot}{\otimes}
\nc{\sot}{{\scriptstyle{\ot}}}
\nc{\otm}{\overline{\ot}}
\nc{\ola}[1]{\stackrel{#1}{\la}}%${\Bbb Z}$

\nc{\scs}[1]{\scriptstyle{#1}} \nc{\mrm}[1]{{\rm #1}}

\nc{\dirlim}{\displaystyle{\lim_{\longrightarrow}}\,}
\nc{\invlim}{\displaystyle{\lim_{\longleftarrow}}\,}

\nc{\bfk}{{\bf k}} \nc{\bfone}{{\bf 1}}
\nc{\rpr}{\circ}
%\nc{\apr}{\cdot}
\nc{\dpr}{{\tiny\diamond}}
\nc{\rprpm}{{\rpr}}

%%%%%%%%%%%%%%%%%%%%% roman fonts, in alphabetic order
\nc{\mmbox}[1]{\mbox{\ #1\ }} \nc{\ann}{\mrm{ann}}
\nc{\Aut}{\mrm{Aut}} \nc{\can}{\mrm{can}}
\nc{\twoalg}{{two-sided algebra}\xspace}
\nc{\colim}{\mrm{colim}}
\nc{\Cont}{\mrm{Cont}} \nc{\rchar}{\mrm{char}}
\nc{\cok}{\mrm{coker}} \nc{\dtf}{{R-{\rm tf}}} \nc{\dtor}{{R-{\rm
tor}}}
\renewcommand{\det}{\mrm{det}}
\nc{\depth}{{\mrm d}}
\nc{\End}{\mrm{End}} \nc{\Ext}{\mrm{Ext}}
\nc{\Fil}{\mrm{Fil}} \nc{\Frob}{\mrm{Frob}} \nc{\Gal}{\mrm{Gal}}
\nc{\GL}{\mrm{GL}} \nc{\Hom}{\mrm{Hom}} \nc{\hsr}{\mrm{H}}
\nc{\hpol}{\mrm{HP}}  \nc{\id}{\mrm{id}} \nc{\im}{\mrm{im}}

\nc{\incl}{\mrm{incl}} \nc{\length}{\mrm{length}}
\nc{\LR}{\mrm{LR}} \nc{\mchar}{\rm char} \nc{\NC}{\mrm{NC}}
\nc{\mpart}{\mrm{part}} \nc{\pl}{\mrm{PL}}
\nc{\ql}{{\QQ_\ell}} \nc{\qp}{{\QQ_p}}
\nc{\rank}{\mrm{rank}} \nc{\rba}{\rm{RBA }} \nc{\rbas}{\rm{RBAs }}
\nc{\rbpl}{\mrm{RBPL}}
\nc{\rbw}{\rm{RBW }} \nc{\rbws}{\rm{RBWs }} \nc{\rcot}{\mrm{cot}}
\nc{\rest}{\rm{controlled}\xspace}
\nc{\rdef}{\mrm{def}} \nc{\rdiv}{{\rm div}} \nc{\rtf}{{\rm tf}}
\nc{\rtor}{{\rm tor}} \nc{\res}{\mrm{res}} \nc{\SL}{\mrm{SL}}
\nc{\Spec}{\mrm{Spec}} \nc{\tor}{\mrm{tor}} \nc{\Tr}{\mrm{Tr}}
\nc{\mtr}{\mrm{sk}}

%%%%%%%%%%%%%%%%%% bold face
\nc{\ab}{\mathbf{Ab}} \nc{\Alg}{\mathbf{Alg}}

%%%%%%%%%%%%%%%%%%%Bbb fonts
\nc{\BA}{{\mathbb A}} \nc{\CC}{{\mathbb C}} \nc{\DD}{{\mathbb D}}
\nc{\EE}{{\mathbb E}} \nc{\FF}{{\mathbb F}} \nc{\GG}{{\mathbb G}}
\nc{\HH}{{\mathbb H}} \nc{\LL}{{\mathbb L}} \nc{\NN}{{\mathbb N}}
\nc{\QQ}{{\mathbb Q}} \nc{\RR}{{\mathbb R}} \nc{\BS}{{\mathbb{S}}} \nc{\TT}{{\mathbb T}}
\nc{\VV}{{\mathbb V}} \nc{\ZZ}{{\mathbb Z}}

%%%%%%%%%%%%%%%%%%% cal fonts

\nc{\calao}{{\mathcal A}} \nc{\cala}{{\mathcal A}}
\nc{\calc}{{\mathcal C}} \nc{\cald}{{\mathcal D}}
\nc{\cale}{{\mathcal E}} \nc{\calf}{{\mathcal F}}
\nc{\calfr}{{{\mathcal F}^{\,r}}} \nc{\calfo}{{\mathcal F}^0}
\nc{\calfro}{{\mathcal F}^{\,r,0}} \nc{\oF}{\overline{F}}
\nc{\calg}{{\mathcal G}} \nc{\calh}{{\mathcal H}}
\nc{\cali}{{\mathcal I}} \nc{\calj}{{\mathcal J}}
\nc{\call}{{\mathcal L}} \nc{\calm}{{\mathcal M}}
\nc{\caln}{{\mathcal N}} \nc{\calo}{{\mathcal O}}
\nc{\calp}{{\mathcal P}} \nc{\calq}{{\mathcal Q}} \nc{\calr}{{\mathcal R}}
\nc{\calt}{{\mathcal T}} \nc{\caltr}{{\mathcal T}^{\,r}}
\nc{\calu}{{\mathcal U}} \nc{\calv}{{\mathcal V}}
\nc{\calw}{{\mathcal W}} \nc{\calx}{{\mathcal X}}
\nc{\CA}{\mathcal{A}}

%%%%%%%%%%%%%%%%%%  frak fonts
\nc{\fraka}{{\mathfrak a}} \nc{\frakB}{{\mathfrak B}}
\nc{\frakb}{{\mathfrak b}} \nc{\frakd}{{\mathfrak d}}
\nc{\oD}{\overline{D}}
\nc{\frakF}{{\mathfrak F}} \nc{\frakg}{{\mathfrak g}}
\nc{\frakm}{{\mathfrak m}} \nc{\frakM}{{\mathfrak M}}
\nc{\frakMo}{{\mathfrak M}^0} \nc{\frakp}{{\mathfrak p}}
\nc{\frakS}{{\mathfrak S}} \nc{\frakSo}{{\mathfrak S}^0}
\nc{\fraks}{{\mathfrak s}} \nc{\os}{\overline{\fraks}}
\nc{\frakT}{{\mathfrak T}}
\nc{\oT}{\overline{T}}
%\nc{\frakx}{{\mathfrak x}}
\nc{\frakX}{{\mathfrak X}} \nc{\frakXo}{{\mathfrak X}^0}
\nc{\frakx}{{\mathbf x}}
%\nc{\frakTxo}{{\frakTx}^0}
\nc{\frakTx}{\frakT}      %All rooted trees, correspond to \ncsha(X)
\nc{\frakTa}{\frakT^a}        % rooted trees for \ncsha(A)
\nc{\frakTxo}{\frakTx^0}   % rooted trees for \ncshao(X)
\nc{\caltao}{\calt^{a,0}}   % rooted trees for \ncshao(A)
\nc{\ox}{\overline{\frakx}} \nc{\fraky}{{\mathfrak y}}
\nc{\frakz}{{\mathfrak z}} \nc{\oX}{\overline{X}}

%%%%%%%%%%%%%%%%%%%%%%%%%%%%%%%%%%%%%%%%%%%%%%%%%%%%%%%%%%%%%%%%%%

\title[Gel'fand-Dorfman bialgebras and Lie conformal bialgebras]{A bialgebra theory of Gel'fand-Dorfman algebras with applications to Lie conformal bialgebras}

\author{Yanyong Hong}
\address{School of Mathematics, Hangzhou Normal University,
Hangzhou 311121, PR China}
\email{yyhong@hznu.edu.cn}

\author{Chengming Bai}
\address{Chern Institute of Mathematics \& LPMC, Nankai University, Tianjin 300071, PR China}
\email{baicm@nankai.edu.cn}

\author{Li Guo}
\address{Department of Mathematics and Computer Science, Rutgers University, Newark, NJ 07102, USA}
         \email{liguo@rutgers.edu}

\subjclass[2020]{
17A30, % Nonassociative algebras satisfying other identities
17B62, %  Lie bialgebras; Lie coalgebras
17B69, % Vertex operators; vertex operator algebras and related structures
17B38, % Yang-Baxter equations and Rota-Baxter operators
17D25,  % Lie-admissible algebras
81T40	% Two-dimensional field theories, conformal field theories, etc. in quantum mechanics
}

\keywords{Gel'fand-Dorfman algebra, Novikov algebra, Gel'fand-Dorfman bialgebra, Novikov bialgebra, Yang-Baxter equation, Lie conformal bialgebra, $\mathcal{O}$-operator, Zinbiel algebra}

\begin{abstract}
Gel'fand-Dorfman algebras (\gdas) give a natural construction of Lie conformal algebras and are in turn characterized by this construction. In this paper, we define the Gel'fand-Dorfman bialgebra (\gdba) and enrich the above construction to a construction of Lie conformal bialgebras by \gdbas. As a special case, Novikov bialgebras yield Lie conformal bialgebras. We further introduce the notion of the
Gel'fand-Dorfman Yang-Baxter equation (GDYBE), whose
skew-symmetric solutions produce \gdbas. Moreover, the
notions of $\mathcal{O}$-operators on \gdas and
pre-Gel'fand-Dorfman  algebras (pre-\gdas) are introduced to
provide skew-symmetric solutions of the GDYBE.
The relationships between these notions for \gdas and the
corresponding ones for Lie conformal algebras are given. In
particular, there is a natural construction of Lie conformal
bialgebras from pre-\gdas. Finally, \gdbas are
characterized by certain matched pairs and  Manin triples of
\gdas.
\end{abstract}

\maketitle

\vspace{-1.2cm}

\tableofcontents

\vspace{-1.2cm}

\allowdisplaybreaks

\section{Introduction}
Lie conformal algebras were introduced by V. Kac \mcite{K1} to give an axiomatic description of singular part of the operator product expansion of chiral fields in conformal field theory and have close connections with vertex algebras \mcite{K1},  infinite-dimensional Lie algebras satisfying the locality property \mcite{K} and Hamiltonian formalism in the theory of nonlinear evolution equations \mcite{BDK}.

In the spirit of the Lie bialgebra theory of Drinfeld and others~\mcite{D}, a bialgebra theory of Lie conformal algebras was developed by J. Liberati in \mcite{L}, together with conformal Manin triples and conformal classical Yang-Baxter equation (CCYBE). The operator form of the CCYBE was investigated in \mcite{HB}, where the definition of $\mathcal{O}$-operators on Lie conformal algebras was introduced. It was shown that $\mathcal{O}$-operators on Lie conformal algebras provide skew-symmetric solutions of the CCYBE,  and left-symmetric conformal algebras introduced in \mcite{HL} naturally provide $\mathcal{O}$-operators on the sub-adjacent Lie conformal algebras. These connections can be summarized in the diagram
\vspb
\begin{equation}
    \begin{split}
        \xymatrix{
            \txt{ \tiny left-symmetric\\ \tiny conformal\\ \tiny algebras} \ar[r]     & \txt{ \tiny $\mathcal{O}$-operators on\\ \tiny Lie conformal\\ \tiny algebras}\ar[r] &
            \text{solutions of}\atop \text{the CCYBE} \ar[r]  & \text{Lie conformal }\atop \text{ bialgberas}  \ar@{<->}[r] &
            \text{conformal Manin triples of} \atop \text{Lie conformal algebras} }
        \vspace{-.1cm}
    \end{split}
    \mlabel{eq:Lieconfdiag}
\end{equation}

On the other hand, the Novikov algebra is an important nonassociative algebra, introduced in connection with Hamiltonian
operators in the formal variational calculus~\mcite{GD1, GD2} and
Poisson brackets of hydrodynamic type~\mcite{BN}.
%The term Novikov algebra was proposed in \mcite{O1}.
It is also a subclass of the pre-Lie algebra which is closely related to numerous areas in
mathematics and physics such as  affine manifolds and affine
structures on Lie groups \mcite{Ko},  convex homogeneous cones \mcite{V}, deformation of
associative algebras \mcite{Ger} and vertex algebras \mcite{BK, BLP}.
A Gel'fand-Dormfan algebra (GD algebra) is a vector space equipped with a Novikov product and a Lie bracket together with a compatibility condition between the two products.\footnote{The term Gel'fand-Dorfman bialgebra was used when the structure was first introduced in \mcite{X1}, but the present name was used in \mcite{KSO} in order to avoid confusion with the usual notion of a bialgebra as an algebra equipped with a coproduct.} Note that a Novikov algebra is a \gda with the trivial Lie algebra structure.  A \gda corresponds to a Hamiltonian pair, which plays a fundamental role in completely integrable systems (see \mcite{GD1}). \gdas related to differential Poisson algebras were studied in \mcite{KSO}, and the variety of commutative \gdas in which the Novikov algebras are commutative was shown to coincide with the variety of transposed Poisson algebras (see \mcite{Sa}).

\gdas naturally give rise to Lie conformal algebras~\mcite{X1} (see Proposition~\mref{Correspond-Algebra}). This construction of Lie conformal algebras in turn determines the axioms of \gdas:
\vspb
\begin{equation}
    \begin{split}
        \xymatrix{
            \text{\gdas} \ar@{<->}[r] &
            \text{a class of Lie conformal algebras}  }
    \end{split}
    \mlabel{eq:algdiag}
\end{equation}
Quite much is known about \gdas. Other than the study on Novikov algebras such as the classification of Novikov algebras in low dimensions \mcite{BM}, constructions of \gdas from more commonly known algebra structures were given in \mcite{X1}. \gdas on certain infinite-dimensional Lie algebras or Novikov algebras were classified in \mcite{OZ,X1}. Thus Diagram~\meqref{eq:algdiag} gives an effective way to produce Lie conformal algebras.

In this paper, we lift Diagram (\mref{eq:algdiag}) to the level of bialgebras by developing a bialgebra theory of Gel'fand-Dorfman algebras (\gdas), which is a combination of a Lie bialgebra \mcite{CP} and a Novikov bialgebra \mcite{HBG} with some compatibility conditions. We show that \gdbas naturally correspond to a class of Lie conformal bialgebras and this correspondence in fact characterizes \gdbas. Since Novikov bialgebras are special \gdbas, Novikov bialgebras also correspond to a (smaller) class of Lie conformal bialgebras. Furthermore, as the bialgebra theories in the classical cases of Lie bialgebras and infinitesimal bialgebras, \gdbas are characterized by some matched pairs and Manin triples of \gdas.

In connection with Diagram~\meqref{eq:Lieconfdiag}, we further introduce the Gel'fand-Dorfman Yang-Baxter equation  (GDYBE) and show that skew-symmetric solutions of the GDYBE
naturally provide \gdbas. The notion of $\mathcal{O}$-operators on \gdas is introduced to give skew-symmetric solutions of the GDYBE, while  pre-Gel'fand-Dorfman algebras (pre-\gdas) are introduced to provide $\mathcal{O}$-operators on the associated \gdas. Furthermore, the connection between \gdas and Lie conformal algebras in Diagram~\meqref{eq:algdiag} enriches to each of the induced notions of \gdas in Diagram \meqref{eq:Lieconfdiag}, yielding the following commutative diagram.
 \vspc
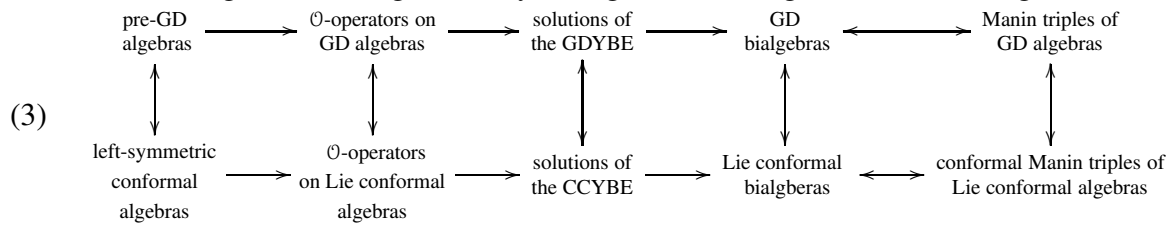
\begin{equation}
    \begin{split}\xymatrix{
            \text{pre-GD}\atop\text{ algebras} \ar[r]  \ar@{<->}[d]     & \mathcal{O}\text{-operators on}\atop\text{\gdas}\ar[r]  \ar@{<->}[d] &
            \text{solutions of}\atop \text{the GDYBE} \ar[r]  \ar@{<->}[d] & \text{GD}\atop \text{ bialgebras}  \ar@{<->}[r]  \ar@{<->}[d] &
        \text{Manin triples of} \atop \text{\gdas} \ar@{<->}[d] \\
        \txt{\tiny left-symmetric\\ \tiny conformal\\ \tiny algebras} \ar[r]     & \txt{ \tiny $\mathcal{O}$-operators \\ \tiny on Lie conformal \\ \tiny algebras}\ar[r] &
            \text{solutions of}\atop \text{the CCYBE} \ar[r]  & \text{Lie conformal }\atop \text{ bialgberas}  \ar@{<->}[r] &
        \text{conformal Manin triples of} \atop \text{Lie conformal algebras}
        }
\vspc
 \end{split}
    \mlabel{eq:bigcommdiag}
\end{equation}
Thanks to this diagram, there is a natural construction of Lie conformal bialgebras from pre-\gdas. Since Novikov bialgebras are special \gdbas, the commutative diagram also holds in the text of Novikov algebras.

Note that the Koszul dual of the operad of \gdas is the operad of
differential right Novikov-Poisson algebras (see \cite{KSO}).
Therefore, by \cite{GK}, there is a natural Lie algebra structure
on the vector space of tensor product of a \gda and a differential
right Novikov-Poisson algebra. Enriching this construction to the
bialgebra level and generalizing the construction of
infinite-dimensional Lie bialgebras developed in \cite{HBG} from
Novikov bialgebras should lead to a construction of more
infinite-dimensional Lie bialgebras, obtained by \gdbas and a quadratic
version of differential right Novikov-Poisson algebras, that is, differential right Novikov-Poisson algebras with nondgenerate
symmetric bilinear forms satisfying certain invariant
conditions.

The paper is outlined as follows. In Section~\mref{s:bialg}, we introduce the notion of \gdbas and apply it to construct a class of Lie conformal bialgebras.

In Section~\mref{s:ybe}, we define the GDYBE and $\calo$-operators for \gdas as the operator form of the GDYBE. We show that skew-symmetric solutions of the GDYBE correspond to a class of solutions of the CCYBE. The relationship between $\mathcal{O}$-operators on \gdas and those on the corresponding Lie conformal algebras is established. Moreover, we show that there is a natural construction of Lie conformal bialgebras from pre-\gdas.

In Section~\mref{s:char}, two characterizations of \gdbas are
given, in terms of matched pairs and Manin triples of \gdas
respectively. Moreover, the correspondence between \gdbas and a class of Lie conformal bialgebras remains valid
under the characterizations of the two types of bialgebras in
terms of the corresponding matched pairs and Manin triples.

\noindent
{\bf Notations.} Let ${\bf k}$ be a field of characteristic $0$ and let $\bfk[\partial]$ be the polynomial algebra with variable $\partial$. All vector spaces and algebras over ${\bf k}$ (resp. all ${\bf k}[\partial]$-modules and Lie conformal algebras) are assumed to be finite-dimensional (resp. finite) unless otherwise stated, even though many results still hold in the infinite-dimensional (resp. infinite) cases. For a vector space $A$, let
\vspc
$$\tau:A\otimes A\rightarrow A\otimes A,\quad a\otimes b\mapsto b\otimes a,\;\;\;a,b\in A,
\vspb
$$
be the flip operator. The identity map is denoted by $\id$. For a vector space $V$,
we use $V[\lambda]$ to denote the set of polynomials of $\lambda$ with coefficients in $V$.
Let $A$ be a vector space with a binary operation $\circ$.
Define linear maps
$L_{\circ}, R_{\circ}:A\rightarrow {\rm End}_{\bf k}(A)$ respectively by
\begin{eqnarray*}
L_{\circ}(a)b:=a\circ b,\;\; R_{\circ}(a)b:=b\circ a, \;\;\;a, b\in A.
\end{eqnarray*}
In particular, when $(A,\circ=[\cdot,\cdot])$ is a Lie algebra, we use the usual notion of the adjoint operator ${\rm ad}(a)(b):=[a,b]$ for all $a,b\in A$.
\vspb
\section{\gdbas corresponding to a class of Lie conformal bialgebras}
\mlabel{s:bialg}
In this section, we introduce the notion of Gel'fand-Dorfman bialgebras (\gdbas) and show that they correspond to a class of Lie conformal bialgebras. Examples of Lie conformal bialgebras obtained from \gdbas are presented.
\vspb
\subsection{A correspondence between GD (co)algebras and Lie conformal (co)algebras}
Recall that
 a {\bf Novikov algebra} $(A, \circ)$ is a vector space $A$ with a binary
    operation $\circ$ satisfying
\begin{eqnarray}
        \mlabel{lef}
        (a\circ b)\circ c-a\circ (b\circ c)&=&(b\circ a)\circ c-b\circ (a\circ c),\\
        \mlabel{Nov}
        (a\circ b)\circ c&=&(a\circ c)\circ b, \;\; a, b,c\in A.
\end{eqnarray}
%
%\begin{rmk}
If $(A, \circ)$ only satisfies Eq. (\mref{lef}), then $(A, \circ)$
is called a {\bf pre-Lie algebra}, also commonly known as a {\bf left-symmetric algebra}.
%\end{rmk}
\begin{defi}\mcite{X1}
Let $(A, \circ)$ be a Novikov algebra and $(A, [\cdot, \cdot])$ be a Lie algebra. If they also satisfy the following compatibility condition
\begin{eqnarray}
[a\circ b, c]-[a\circ c,b]+[a,b]\circ c-[a, c]\circ b-a\circ [b,c]=0,\;\;\text{ $a$, $b$, $c\in A,$}
\end{eqnarray}
then $(A, \circ, [\cdot, \cdot])$ is called a {\bf Gel'fand-Dorfman algebra (\gda)}.
\end{defi}

\begin{rmk}
Note that such an algebra is called a Gel'fand-Dorfman bialgebra in \mcite{X1}. To avoid confusion with the notion of the usual bialgebra consisting of an algebra and a coalgebra, we call it a Gel'fand-Dorfman algebra following \mcite{KSO}. We also point out that a Novikov algebra is defined to be a \gda with the trivial Lie algebra structure.
\end{rmk}

We next recall the notion of Lie conformal algebras.
\begin{defi}\mcite{K1}
A {\bf Lie conformal algebra} $R$ is a ${\bf k}[\partial]$-module with a $\lambda$-bracket $[\cdot_\lambda \cdot]$ which defines a ${\bf k}$-bilinear
map $R\times R\rightarrow R[\lambda]$ satisfying
\begin{eqnarray*}
&& \textbf{(conformal sesquilinearity)}\ \ [\partial a_\lambda b]=-\lambda [a_\lambda b],~~~[a_\lambda \partial b]=(\lambda+\partial)[a_\lambda b],\\
&&\textbf{(skew-symmetry)}\ [a_\lambda b]=-[b_{-\lambda-\partial}a],\\
&&\textbf{(Jacobi identity)}\ \ [a_\lambda[b_\mu c]]=[[a_\lambda b]_{\lambda+\mu} c]+[b_\mu[a_\lambda c]],\quad a, b, c\in R.
\end{eqnarray*}
A Lie conformal algebra $R$ is called {\bf finite}, if it is finitely generated as a ${\bf k}[\partial]$-module.
\end{defi}

There is a natural correspondence between \gdas and Lie conformal algebras which also characterizes \gdas.

\begin{pro}\mcite{X1}
Let $(A, \circ, [\cdot,\cdot])$ be a $\bfk$-vector space with binary operations $\circ$ and $[\cdot,\cdot]$. Equip the free ${\bf k}[\partial]$-module $R:={\bf k}[\partial]A(=\bfk[\partial]\ot_\bfk A)$ with the bilinear map
\begin{eqnarray*}
[\cdot_\lambda \cdot]: R\times R \to R[\lambda], \quad [a_\lambda b]:=\partial(b\circ a)+\lambda(a\star b)+[a,b],\;\; a, b\in A,
\end{eqnarray*}
where $a\star b=a\circ b+b\circ a$. Then $(R,[\cdot_\lambda\cdot])$ is a Lie conformal algebra if and only if $(A, \circ, [\cdot,\cdot])$ is a \gda. We call
$R$ the {\bf Lie conformal algebra corresponding to (the \gda) $(A,
\circ, [\cdot,\cdot])$}. In particular, if $(A, [\cdot,\cdot])$ is abelian, we say that $R$ is the {\bf Lie conformal algebra corresponding to (the Novikov algebra) $(A, \circ)$}.
\mlabel{Correspond-Algebra}
\end{pro}

Recall \mcite{HBG, KUZ} that a {\bf Novikov coalgebra} $(A, \Delta)$ is a vector space $A$ with a linear map $\Delta: A\rightarrow A\otimes A$ such that
\begin{eqnarray}
\mlabel{Lc3}(\id\otimes \Delta)\Delta(a)-(\tau\otimes {\rm id})(\id\otimes \Delta)\Delta(a)&=&(\Delta\otimes \id)\Delta(a)-(\tau\otimes {\rm id})(\Delta\otimes \id)\Delta(a),\\
\mlabel{Lc4}(\tau\otimes {\rm id})(\id\otimes \Delta)\tau
     \Delta(a)&=&(\Delta\otimes \id)\Delta(a), \;\;a\in A.
\end{eqnarray}

A {\bf Lie coalgebra} $(L, \delta)$ is a vector space $L$ with a linear map $\delta: L\rightarrow L\otimes L$ satisfying
\begin{eqnarray}
&&\delta(a)=-\tau \delta(a),\\
&& (\id \otimes \delta)\delta(a)-(\tau\otimes \id)(\id \otimes \delta)\delta(a)=(\delta\otimes \id)\delta(a),\;\;a\in L.
\end{eqnarray}

Let $\langle \cdot, \cdot \rangle$ be the usual pairing between a vector space and its dual vector space. Let $V$ and $W$ be vector spaces. For a linear map $\varphi: V\rightarrow W$, denote the transpose map by $\varphi^\ast: W^\ast\rightarrow V^\ast$ given by
\vspc
\begin{eqnarray*}
\langle \varphi^\ast(f),v\rangle=\langle f, \varphi(v)\rangle,\;\;\; f\in W^\ast, v\in V.
\end{eqnarray*}
Note that $(A, \Delta)$ is a Novikov coalgebra if and only if $(A^\ast, \Delta^\ast)$ is a Novikov algebra, and $(L, \delta)$ is a Lie coalgebra if and only if $(L^\ast, \delta^\ast)$ is a Lie algebra.

Next, we introduce the notion of a Gel'fand-Dorfman coalgebra.

\begin{defi}
A {\bf Gel'fand-Dorfman coalgebra (GD coalgebra)} is a triple $(A, \Delta, \delta)$ where $(A, \Delta)$ is a Novikov coalgebra, $(A, \delta)$ is a Lie coalgebra, and they satisfy, for all $a\in A$,
\vspb
\begin{eqnarray}
\mlabel{Lc2}&&(\id \otimes \delta)\Delta(a)-(\tau\otimes\id)(\id\otimes \Delta)\delta(a)
+(\tau\otimes\id)(\id \otimes \delta)\tau \Delta(a)=(\Delta\otimes \id)\delta(a)+(\delta\otimes \id)\Delta(a).
\end{eqnarray}
\end{defi}

It is easy to check that
$(A, \Delta, \delta)$ is a GD coalgebra if and only if $(A^\ast, \Delta^\ast, \delta^\ast)$ is a \gda.

\begin{defi}\mcite{L}
A {\bf Lie conformal coalgebra} is a ${\bf k}[\partial]$-module $R$ endowed with a ${\bf k}[\partial]$-module homomorphism $\delta: R\rightarrow R\otimes R$ such that
\vspb
\begin{eqnarray*}
&&\delta(a)=-\tau \delta(a), \\
&&(\id\otimes \delta)\delta(a)-(\tau\otimes \id) (\id\otimes \delta)\delta(a)=(\delta\otimes \id)\delta(a),\;\; a\in R,
\end{eqnarray*}
where the ${\bf k}[\partial]$-module structure on $R\otimes R$ is defined by $\partial(a\otimes b)=\partial a \otimes b+a\otimes \partial b$ for $a$, $b\in R$.
\vspb
\end{defi}

Let $\partial^{\otimes^2}:=\partial\otimes \id+\id\otimes \partial$. A coalgebra analog of Proposition \mref{Correspond-Algebra} is given as follows.

\begin{pro}\mlabel{corr-coalgebra}
Let $R={\bf k}[\partial]A$ be the free ${\bf k}[\partial]$-module on a vector space $A$. Suppose that $\delta_0$ and $\Delta$ are ${\bf k}[\partial]$-module homomorphisms from $R$ to $R\otimes R$ such that $\delta_0(A)\subset A\otimes A$ and $\Delta(A)\subset A\otimes A$.
Define a linear map $\delta: R\rightarrow R\otimes R$ by first taking
\vspb
\begin{eqnarray}\mlabel{cobracket}
\delta(a)=\delta_0(a)+(\partial\otimes \id)\Delta(a)-\tau (\partial\otimes \id)\Delta(a),~~~~a\in A,
\end{eqnarray}
and then extending to $R$. Then $(R, \delta)$ is a Lie conformal coalgebra if and only if $(A, \Delta, \delta_0)$ is a GD coalgebra.
\end{pro}
\vspb
\begin{proof}
Since $\delta_0$ and $\Delta$ are ${\bf k}[\partial]$-module homomorphisms, $\delta$ is a ${\bf k}[\partial]$-module homomorphism.
Therefore, $\delta(a)=-\tau\delta(a)$ for all $a\in R$ if and only if $\delta_0(b)=-\tau\delta_0(b)$ for all $b\in A$.

By the definition of $\delta$, we have
\vspb
\begin{eqnarray*}
(\id \otimes \delta)\delta(a)&=&(\id \otimes \delta)(\delta_0(a)+(\partial\otimes \id)\Delta(a)-\tau (\partial\otimes \id)\Delta(a))\\
&=&(\id\otimes \delta_0)\delta_0(a)+(\id\otimes \partial\otimes \id)(\id\otimes \Delta)\delta_0(a)-(\id\otimes \tau)(\id\otimes \partial\otimes \id)(\id\otimes \Delta)\delta_0(a)\\
&&+(\partial\otimes \id\otimes \id)(\id\otimes \delta_0)\Delta(a)+(\partial\otimes \id\otimes \id)(\id\otimes \partial\otimes \id)(\id\otimes \Delta)\Delta(a)\\
&&-(\partial\otimes \id\otimes \id)(\id\otimes \tau)(\id\otimes \partial\otimes \id)(\id\otimes \Delta)\Delta(a)
-(\id\otimes \partial^{\otimes^2})(\id\otimes \delta_0)\tau \Delta(a)\\
&&-(\id\otimes \partial^{\otimes^2})(\id\otimes \partial \otimes \id)(\id\otimes \Delta)\tau\Delta(a)\\
&&+(\id\otimes \partial^{\otimes^2})(\id\otimes \tau)(\id\otimes \partial\otimes \id)(\id\otimes \Delta)\tau \Delta(a).
\end{eqnarray*}
Similarly, we have
\vspb
\begin{eqnarray*}(\delta\otimes \id)\delta(a)&=&(\delta_0\otimes \id)\delta_0(a)
+(\partial\otimes \id\otimes \id)(\Delta\otimes \id)\delta_0(a)\\
&&-(\tau\otimes \id)(\partial\otimes \id\otimes \id)(\Delta\otimes \id)\delta_0(a)
+(\partial^{\otimes^2}\otimes \id)(\delta_0\otimes \id)\Delta(a)\\
&&+(\partial\otimes \id\otimes \id)(\partial^{\otimes^2}\otimes \id)(\Delta\otimes \id)\Delta(a)
-(\tau\otimes \id)(\partial\otimes \id\otimes \id)(\partial^{\otimes^2}\otimes \id)(\Delta\otimes \id)\Delta(a)\\
&&-(\id\otimes \id\otimes \partial)(\delta_0\otimes \id)\tau \Delta(a)
-(\id\otimes \id\otimes \partial)(\partial\otimes \id\otimes \id)(\Delta\otimes \id)\tau \Delta(a)\\
&&+(\id\otimes \id\otimes \partial)(\id\otimes \partial\otimes \id)(\tau\otimes \id)(\Delta\otimes \id)\tau \Delta(a).
\end{eqnarray*}
Then by the cocommutativity of $\delta$, considering the terms in $A\otimes A\otimes A$, $\partial A\otimes A\otimes A$,  $\partial A\otimes \partial A\otimes A$, and
$\partial^2A\otimes A\otimes A$, the following equation
\vspb
$$(\id\otimes \delta)\delta(a)-(\tau\otimes \id)(\id\otimes \delta)\delta(a)=(\delta\otimes \id)\delta(a),\;\;a\in A,$$ holds if and only if the following equations hold.
\vspb
\begin{eqnarray}
~~&&(\id\otimes \delta_0)\delta_0(a)-(\tau\otimes \id) (\id\otimes \delta_0)\delta_0(a)=(\delta_0\otimes \id)\delta_0(a),\\
\mlabel{v1}~~&&(\id\otimes \delta_0)\Delta(a)-(\tau\otimes \id) (\id\otimes \Delta)\delta_0(a)
+(\tau\otimes \id)(\id\otimes \delta_0)\tau \Delta(a)\\
&&\mbox{}\hspace{0.5cm}=(\Delta\otimes
\id)\delta_0(a)+(\delta_0\otimes \id)\Delta(a),\nonumber\\
\mlabel{v4}~~&&(\id\otimes \Delta)\Delta(a)-(\tau\otimes \id)(\id\otimes \Delta)\Delta(a)=(\Delta\otimes \id)\Delta(a)-(\tau\otimes \id)(\Delta\otimes \id)\Delta(a),\\
\mlabel{v7}~~&&(\tau\otimes \id)(\id\otimes \Delta)\tau \Delta(a)=(\Delta\otimes \id)\Delta(a),
\quad a\in A.
\end{eqnarray}
Then the conclusion follows directly.
\end{proof}

\vspd

\subsection{\gdbas and the relationship with Lie conformal bialgebras}\

Let $(A,\circ, [\cdot,\cdot])$ be a \gda. Set
\vspb
$$a\star b:=a\circ b+b\circ a,\;\;a.(b\otimes c):=[a,b]\otimes c+b\otimes [a,c],\;\;a, b, c\in A.$$

Recall \mcite{HBG} that a {\bf Novikov bialgebra} is a tripe $(A,\circ,\Delta)$ where  $(A,\circ)$ is a Novikov algebra and $(A, \Delta)$ is a Novikov coalgebra satisfying
\vspb
\begin{eqnarray}
            &\mlabel{Lb5}\Delta(a\circ b)=(R_{\circ} (b)\otimes \id)\Delta (a)+(\id\otimes L_{\star}(a))(\Delta(b)+\tau\Delta(b)),&\\
            &\mlabel{Lb6}(L_{\star}(a)\otimes \id)\Delta(b)-(\id\otimes L_{\star}(a))\tau\Delta (b)=(L_{\star}(b)\otimes \id)\Delta(a)-(\id\otimes L_{\star}(b))\tau\Delta(a),&\\
            &\mlabel{Lb7}\quad (\id\otimes R_{\circ}(a)-R_{\circ}(a)\otimes
            \id)(\Delta(b)+\tau\Delta(b))=(\id\otimes R_{\circ}(b)-R_{\circ}(b)\otimes
            \id)(\Delta(a)+\tau\Delta(a)),&
    \end{eqnarray}
 for all $a$, $b\in A$.
A {\bf  Lie bialgebra}
is a triple $(L, [\cdot,\cdot],\delta)$ such that
$(L, [\cdot, \cdot])$ is a Lie algebra, $(L,
\delta)$ is a Lie coalgebra, and the following
compatibility condition holds.
\begin{eqnarray} \mlabel{lia3}
\notag \delta([a,b])=a.\delta(b)-b.\delta(a),\;\;  a, b\in L.
\vspb
\end{eqnarray}

Next, we introduce the notion of a Gel'fand-Dorfman bialgebra.

\begin{defi}
Let $(A, \circ, [\cdot,\cdot])$ be a \gda. If there are linear maps $\delta$, $\Delta: A\rightarrow A\otimes A$ such that $(A, \Delta, \delta)$ is a GD coalgebra and the following compatibility condition holds: 
\vspb
\begin{eqnarray}
\mlabel{Lb4}&&\delta(b\circ a)+\Delta([a,b])=(R_{\circ}(a)\otimes \id)\delta(b)+(L_{\circ}(b)\otimes \id)\delta(a)\\
&&\;\;\;\;\;\;\;\;\;+(\id\otimes L_{\star}(b))\delta(a)+a.\Delta(b)-(\id\otimes \ad(b))(\Delta(a)+\tau\Delta(a)),\;\;a, b\in A,\nonumber
\end{eqnarray}
then $(A, \circ, [\cdot,\cdot], \Delta, \delta)$ is called a {\bf Gel'fand-Dorfman bialgebra (\gdba)}.
\vspb
\end{defi}

\begin{rmk}
Note again that the above definition of \gdbas is different from that in \mcite{X1} which is now simply called a \gda.
On the other hand, if $(A, \circ, [\cdot,\cdot], \Delta, \delta)$ is a \gdba, then $(A, \circ, \Delta)$ is a Novikov bialgebra and $(A, [\cdot,\cdot], \delta)$ is a Lie bialgebra. Moreover, a Lie bialgebra can be regarded as a \gdba in which $\circ$ and $\Delta$ are trivial, and a Novikov bialgebra can be regarded as a \gdba in which  $[\cdot,\cdot]$ and $\delta$ are trivial.
\end{rmk}

We introduce two special classes of  \gdbas with $\Delta$, $[\cdot,\cdot]$ and $\circ$, $\delta$
being trivial respectively.
\begin{defi}
Let $(A, \circ)$ be a Novikov algebra and $(A, \delta)$ be a Lie coalgebra. If
\vspb
\begin{eqnarray*}
\delta(b\circ a)=(R_{\circ}(a)\otimes \id)\delta(b)+(L_{\circ}(b)\otimes \id)\delta(a)+(\id\otimes L_{\star}(b))\delta(a),\;\;a, b\in A,
\end{eqnarray*}
then the resulting \gdba $(A, \circ, \delta)$ (with $\Delta$ and $[\cdot,\cdot]$ being trivial) is called a {\bf \gdba of Novikov type}.

Let $(L, [\cdot, \cdot])$ be a Lie algebra and $(L, \Delta)$ be a Novikov coalgebra. If
\vspb
\begin{eqnarray*}
&&\Delta([a,b])=a.\Delta(b)-(\id\otimes \ad(b))(\Delta(a)+\tau\Delta(a)), \;\;a, b\in A,
\end{eqnarray*}
the resulting \gdba $(L, [\cdot,\cdot], \Delta)$ (with $\circ$ and $\delta$ being trivial) is called a {\bf \gdba of Lie type}.
\vspb
\end{defi}

Here are some examples of the two classes of \gdbas.
\vspb
\begin{ex}\mlabel{ex-GD}
Let $(A={\bf k}e_1\oplus {\bf k}e_2, \circ)$ be the Novikov
algebra given by~\cite{BM}
\vspb
\begin{eqnarray*}
e_1\circ e_1=e_2\circ e_1=e_2\circ e_2=0,\;\;\;e_1\circ e_2=e_2.
\end{eqnarray*}
Define a linear map $\delta_0: A\rightarrow A\otimes A$ by
\vspb
\begin{eqnarray*}
\delta_0(e_1)=0,\;\;\;\delta_0(e_2)=e_1\otimes e_2-e_2\otimes e_1.
\end{eqnarray*}
One can check that $(A, \circ, \delta_0)$ is a \gdba of Novikov type.

Let $(L={\bf k}e_1\oplus {\bf k}e_2, [\cdot,\cdot])$ be the Lie algebra given by
$[e_1,e_2]=e_2.$
Define a linear map $\Delta: A\rightarrow A\otimes A$ by
\vspb
\begin{eqnarray*}
\Delta(e_1)=0,\;\;\;\Delta(e_2)=e_1\otimes e_2.
\end{eqnarray*}
One can check that $(L, [\cdot,\cdot], \Delta)$ is a \gdba of Lie type.
\vspb
\end{ex}

\begin{defi}\mcite{L}
A {\bf Lie conformal bialgebra} is a triple $(R, [\cdot_\lambda \cdot],\delta)$ such that $(R, [\cdot_\lambda \cdot])$ is a Lie conformal algebra,
$(R, \delta)$ is a Lie conformal coalgebra, and they satisfy the cocycle condition
\vspc
\begin{eqnarray*}
a_\lambda \delta(b)-b_{-\lambda-\partial}\delta(a)=\delta([a_\lambda b]),\;\;a, b\in R,
\end{eqnarray*}
where $a_\lambda \delta(b):=\sum ([a_\lambda b_{(1)}]\otimes
b_{(2)}+b_{(1)}\otimes [a_\lambda b_{(2)}])$ if $\delta(b)=\sum
b_{(1)}\otimes b_{(2)}$, for all $a,b\in R$.
\end{defi}

We establish the relationship between \gdbas and Lie conformal bialgebras.
\begin{thm}\mlabel{thm1-corr}
Let $R={\bf k}[\partial]A$ be the Lie conformal algebra corresponding to a \gda $(A, \circ, [\cdot, \cdot])$.
 Suppose that $\delta_0$ and $\Delta$ are ${\bf k}[\partial]$-module homomorphisms from $R$ to $R\otimes R$ such that $\delta_0(A)\subset A\otimes A$ and $\Delta(A)\subset A\otimes A$. Then $(R, [\cdot_\lambda \cdot], \delta)$ is a Lie conformal bialgebra with $\delta: R\rightarrow R\otimes R$ defined by Eq.~\meqref{cobracket}
if and only if $(A, \circ, [\cdot,\cdot], \Delta, \delta_0)$ is a \gdba. We call $(R, [\cdot_\lambda \cdot], \delta)$ {\bf the Lie conformal bialgebra corresponding to the \gdba $(A, \circ, [\cdot,\cdot], \Delta, \delta_0)$}.
\vspb
\end{thm}
\begin{proof}
By Proposition \mref{corr-coalgebra}, $(R, \delta)$ is a Lie conformal coalgebra if and only if $(A, \Delta, \delta_0)$ is a GD coalgebra. Let $a,b\in A$.
Set $\delta_0(a)=\sum a_{(1)}\otimes a_{(2)}$ and $\Delta(a)=\sum \overline{a_{(1)}}\otimes \overline{a_{(2)}}$.
Then we get
\vspb
\begin{eqnarray*}
&&a_\lambda \delta(b)-b_{-\lambda-\partial}\delta(a)\\
&&\quad=\mbox{}\hspace{0.5cm}a_\lambda(\delta_0(b)+(\partial\otimes \id)\Delta(b)-\tau(\partial\otimes \id)\Delta(b))
-b_{-\lambda-\partial}(\delta_0(a)+(\partial\otimes \id)\Delta(a)-\tau(\partial\otimes \id)\Delta(a))\\
&&\quad =\mbox{}\hspace{0.5cm}a_\lambda(\sum b_{(1)}\otimes b_{(2)})+a_\lambda(\sum \partial \overline{b_{(1)}} \otimes \overline{b_{(2)}}-\overline{b_{(2)}}
\otimes \partial \overline{b_{(1)}})\\
&&\mbox{}\hspace{0.7cm}-b_{-\lambda-\partial}(\sum a_{(1)}\otimes a_{(2)})-b_{-\lambda-\partial}(\sum \partial \overline{a_{(1)}} \otimes \overline{a_{(2)}}-\overline{a_{(2)}}
\otimes \partial \overline{a_{(1)}})\\
&&\quad=\mbox{}\hspace{0.5cm}\sum [a_\lambda b_{(1)}]\otimes b_{(2)}
+\sum b_{(1)}\otimes[a_\lambda b_{(2)}]+\sum [a_\lambda (\partial\overline{b_{(1)}})]\otimes \overline{b_{(2)}}
+\sum \partial\overline{b_{(1)}}\otimes [a_\lambda \overline{b_{(2)}}]\\
&&\mbox{}\hspace{0.7cm}-\sum [a_\lambda \overline{b_{(2)}}]\otimes \partial\overline{b_{(1)}}
-\sum \overline{b_{(2)}}\otimes [a_\lambda \partial\overline{b_{(1)}}]
-\sum [b_{-\lambda-\partial^{\otimes^2}} a_{(1)}]\otimes a_{(2)}\\
&&\mbox{}\hspace{0.7cm}-\sum a_{(1)}\otimes [b_{-\lambda-\partial^{\otimes^2}}a_{(2)}]
-\sum [b_{-\lambda-\partial^{\otimes^2}}\partial \overline{a_{(1)}} ]\otimes \overline{a_{(2)}} -\sum \partial \overline{a_{(1)}} \otimes [b_{-\lambda-\partial^{\otimes^2}} \overline{a_{(2)}}]\\
&&\mbox{}\hspace{0.7cm}+\sum [b_{-\lambda-\partial^{\otimes^2}} \overline{a_{(2)}}]\otimes \partial\overline{a_{(1)}}
+\sum \overline{a_{(2)}}\otimes [b_{-\lambda-\partial^{\otimes^2}}\partial \overline{a_{(1)}}]\\
&&\quad=\mbox{}\hspace{0.5cm} \sum (\partial(b_{(1)}\circ a)+\lambda(a\star b_{(1)})+[a,b_{(1)}])\otimes b_{(2)}\\
&&\mbox{}\hspace{0.7cm} +\sum b_{(1)}\otimes (\partial(b_{(2)}\circ a)+\lambda(a\star b_{(2)})+[a, b_{(2)}])\\
&&\mbox{}\hspace{0.7cm}+\sum (\lambda+\partial\otimes \id)(\partial(\overline{b_{(1)}}\circ a)+\lambda(\overline{b_{(1)}}\star a)+[a,\overline{b_{(1)}}])
\otimes \overline{b_{(2)}}\\
&&\mbox{}\hspace{0.7cm}+\sum \partial \overline{b_{(1)}}\otimes (\partial(\overline{b_{(2)}}\circ a)+\lambda(\overline{b_{(2)}}\star a)
+[a,\overline{b_{(2)}}])\\
&&\mbox{}\hspace{0.7cm}-\sum (\partial(\overline{b_{(2)}}\circ a)+\lambda(\overline{b_{(2)}}\star a)+[a,\overline{b_{(2)}}])\otimes \partial\overline{b_{(1)}}\\
&&\mbox{}\hspace{0.7cm}-\sum \overline{b_{(2)}}\otimes (\lambda+\partial)(\partial(\overline{b_{(1)}}\circ a)+\lambda(\overline{b_{(1)}}\star a)+[a,\overline{b_{(1)}}])\\
&&\mbox{}\hspace{0.7cm}-\sum (\partial (a_{(1)}\circ b)+(-\lambda-\partial^{\otimes^2})(a_{(1)}\star b)+[b, a_{(1)}])\otimes a_{(2)}\\
&&\mbox{}\hspace{0.7cm}-\sum a_{(1)}\otimes (\partial(a_{(2)}\circ b)+(-\lambda-\partial^{\otimes^2})(a_{(2)}\star b)+[b, a_{(2)}])\\
&&\mbox{}\hspace{0.7cm}-\sum(-\lambda-\id\otimes \partial)(\partial(\overline{ a_{(1)}}\circ b)+(-\lambda-\partial^{\otimes^2})(\overline{a_{(1)}}\star b)
+[b,\overline{a_{(1)}}])\otimes \overline{a_{(2)}}\\
&&\mbox{}\hspace{0.7cm}-\sum \partial\overline{a_{(1)}} \otimes(\partial (\overline{a_{(2)}}\circ b)
+(-\lambda-\partial^{\otimes^2})(\overline{a_{(2)}}\star b)+[b,\overline{a_{(2)}}])\\
&&\mbox{}\hspace{0.7cm}+\sum (\partial(\overline{a_{(2)}}\circ b)+(-\lambda-\partial^{\otimes^2})(\overline{a_{(2)}}\star b)+[b,\overline{a_{(2)}}])\otimes \partial \overline{a_{(1)}}\\
&&\mbox{}\hspace{0.7cm}+\sum  (-\lambda-\partial\otimes \id)(\overline{a_{(2)}}\otimes (\partial( \overline{ a_{(1)}}\circ b)+(-\lambda-\partial^{\otimes^2})
(\overline{ a_{(1)}}\star b)+[b,\overline{ a_{(1)}}])),
\end{eqnarray*}
and
\vspc
\begin{eqnarray*}
\delta([a_\lambda b]&=& \delta(\partial(b\circ a)+\lambda(b\star a)+[a,b])\\
&=&\partial^{\otimes^2}\delta(b\circ a)+\lambda \delta(b\star a)+\delta([a,b])\\
&=& \partial^{\otimes^2}(\delta_0(b\circ a)+(\partial\otimes \id)\Delta(b\circ a)-\tau(\partial\otimes \id)\Delta(b\circ a))\\
&&+\lambda(\delta_0(b\star a)+(\partial\otimes \id)\Delta(b\star a)-\tau(\partial\otimes \id)\Delta(b\star a))\\
&&+\delta_0([a,b])+(\partial\otimes \id)\Delta([a,b])-\tau(\partial\otimes \id)\Delta([a,b]).
\end{eqnarray*}

Since $\delta=-\tau \delta$, we only need to consider the coefficients of $A\otimes A$, $\partial A\otimes A$, $\partial^2 A\otimes A$, $\partial A\otimes \partial A$, $\lambda A\otimes A$, $\lambda(\partial A\otimes A)$ and $\lambda^2 A\otimes A$. By comparing the coefficients of those terms in these vector spaces respectively, we obtain that the equation
\vspb
$$a_\lambda \delta(b)-b_{-\lambda-\partial}\delta(a)=\delta([a_\lambda b]), \;\;a, b\in A,$$
holds if and only if the following equations hold 
\begin{eqnarray}
\mlabel{lbc1}\;\;&&\delta_0([a,b])=a.\delta_0(b)-b.\delta_0(a),\\
\mlabel{lbc2}&&\delta_0(b\circ a)+\Delta([a,b])=\sum(b_{(1)}\circ a\otimes b_{(2)})+\sum [a,\overline{ b_{(1)}}]\otimes \overline{ b_{(2)}}
+\sum \overline{ b_{(1)}}\otimes [a,\overline{ b_{(2)}}]\\
&&\;\;\;\;+\sum (b\circ a_{(1)})\otimes a_{(2)}+\sum a_{(1)}\otimes a_{(2)}\star b-\sum \overline{ a_{(1)}}\otimes [b,\overline{ a_{(2)}}]
-\sum \overline{ a_{(2)}}\otimes [b,\overline{ a_{(1)}}],\nonumber\\
\mlabel{lbc3}&&\Delta(b\circ a)=\sum \overline{ b_{(1)}}\circ a\otimes \overline{ b_{(2)}}+\overline{ a_{(1)}}\otimes (\overline{ a_{(2)}}\star b)+\sum \overline{ a_{(2)}}\otimes
\overline{ a_{(1)}}\star b,\\
\mlabel{lbc4}&&\Delta(b\circ a)-\tau\Delta(b\circ a)=\sum \overline{ b_{(1)}}\otimes \overline{ b_{(2)}}\circ a-\sum \overline{ b_{(2)}}\circ a \otimes \overline{ b_{(1)}}
-\sum (b\circ \overline{ a_{(1)}})\otimes \overline{ a_{(2)}}\\
&&\;\;\;\;+\sum \overline{ a_{(1)}}\otimes b\circ \overline{ a_{(2)}}-\sum (b\circ \overline{ a_{(2)}})\otimes \overline{ a_{(1)}}+\sum \overline{ a_{(2)}}\otimes b\circ \overline{ a_{(1)}},\nonumber\\
\mlabel{lbc5}&&\delta_0(b\star a)=\sum (a\star b_{(1)})\otimes b_{(2)}+\sum b_{(1)}\otimes (a\star b_{(2)})+\sum [a,\overline{ b_{(1)}}]\otimes \overline{ b_{(2)}}\\
&&\;\;\;\;-\sum \overline{ b_{(2)}}\otimes [a,\overline{ b_{(1)}}]+\sum a_{(1)}\star b \otimes a_{(2)}+\sum a_{(1)}\otimes (a_{(2)}\star b)\nonumber\\
&&\;\;\;\;+\sum [b,\overline{ a_{(1)}}]\otimes \overline{ a_{(2)}}-\sum \overline{ a_{(2)}}\otimes [b,\overline{ a_{(1)}}],\nonumber\\
\mlabel{lbc6}&&\Delta(b\star a)=\sum \overline{ b_{(1)}}\circ a\otimes \overline{ b_{(2)}}+\sum (\overline{ b_{(1)}}\star a)\otimes \overline{ b_{(2)}}
+\sum \overline{ b_{(1)}}\otimes \overline{ b_{(2)}}\star a\\
&&\;\;\;\;+\sum \overline{ a_{(1)}}\circ b \otimes \overline{ a_{(2)}}-\sum \overline{a_{(1)}}\star b \otimes \overline{ a_{(2)}}+\sum \overline{ a_{(1)}}\otimes \overline{ a_{(2)}}\star b\nonumber\\
&&\;\;\;\;+\sum \overline{ a_{(2)}}\otimes \overline{ a_{(1)}}\star b+\overline{ a_{(2)}}\otimes \overline{ a_{(1)}}\star b,\nonumber\\
&&\mlabel{lbc7}\sum(\overline{ b_{(1)}}\star a)\otimes \overline{ b_{(2)}}-\sum \overline{ b_{(2)}}\otimes \overline{ b_{(1)}}\star a=\sum(\overline{ a_{(1)}}\star b)\otimes \overline{ a_{(2)}}-\sum \overline{ a_{(2)}}\otimes \overline{ a_{(1)}}\star b,
\end{eqnarray}
for all $a, b\in A$. It is straightforward to show
that Eqs. (\mref{lbc3}), (\mref{lbc4}), (\mref{lbc6}) and
(\mref{lbc7}) hold if and only if Eqs. (\mref{Lb5})-(\mref{Lb7})
hold, and Eqs. (\mref{lbc2}) and (\mref{lbc5}) hold if and only if
Eq. (\mref{Lb4}) holds. This completes the proof. \vspb
\end{proof}

Theorem \mref{thm1-corr} has the following direct consequence.
\vspb
\begin{cor}
\mlabel{consthm1}
Let $R={\bf k}[\partial]A$ be the Lie conformal algebra corresponding to a Novikov algebra $(A, \circ)$. Let $\Delta: A\rightarrow  A\otimes A$ be a linear map.
Then $(R, [\cdot_\lambda \cdot], \delta)$ is a Lie conformal bialgebra with the ${\bf k}[\partial]$-module homomorphism $\delta: R\rightarrow R\otimes R$ given by
\vspb
\begin{eqnarray}
\delta(a)=(\partial\otimes \id)\Delta(a)-\tau (\partial\otimes \id)\Delta(a),\;\;\;\;a\in A,
\end{eqnarray}
if and only if $(A, \circ, \Delta)$ is a Novikov bialgebra.
\end{cor}

We end the section with several examples of Lie conformal bialgebras corresponding to \gdbas by Theorem \mref{thm1-corr}.
\begin{ex}
The Lie conformal bialgebra $(R={\bf k}[\partial]A={\bf
k}[\partial]e_1\oplus {\bf k}[\partial]e_2, [\cdot_\lambda \cdot],
\delta)$ corresponding to the \gdba $(A, \circ, \delta_0)$ in
Example \mref{ex-GD} is given as follows.
\begin{eqnarray*}
&&[{e_1}_\lambda e_1]=[{e_2}_\lambda e_2]=0,\;\;\;[{e_1}_\lambda e_2]=\lambda e_2,\\
&&\delta(e_1)=0, \;\;\; \delta(e_2)=e_1\otimes e_2-e_2\otimes e_1.
\end{eqnarray*}

Similarly, the Lie
conformal bialgebra $(R={\bf k}[\partial]L={\bf
k}[\partial]e_1\oplus {\bf k}[\partial]e_2, [\cdot_\lambda \cdot],
\delta)$ corresponding to the \gdba $(L, [\cdot,\cdot], \Delta)$
in Example \mref{ex-GD} is given as follows. 
\begin{eqnarray*}
&&[{e_1}_\lambda e_1]=[{e_2}_\lambda e_2]=0,\;\;\;[{e_1}_\lambda e_2]=e_2,\\
&&\delta(e_1)=0, \;\;\; \delta(e_2)=\partial e_1\otimes e_2-e_2\otimes \partial e_1.
\end{eqnarray*}
\vspb
\end{ex}
We also give examples of Lie conformal bialgebras from Novikov bialgebras.
\begin{ex}
Let $(A={\bf k}e_1\oplus {\bf k}e_2, \circ, \Delta)$ be the Novikov bialgebra from  \mcite{HBG}, defined by
\vspb
\begin{eqnarray*}
&&e_1\circ e_1=e_1,\;\;e_2\circ e_1=e_2,\;\;e_1\circ e_2=e_2\circ e_2=0,\;\;\Delta(e_1)=e_2\otimes e_2, \;\; \Delta(e_2)=0.
\end{eqnarray*}
Then by Corollary \mref{consthm1}, there is a Lie conformal bialgebra structure on
$R={\bf k}[\partial]A$ given by
\begin{eqnarray*}
&&[{e_1}_\lambda e_1]=(\partial+2\lambda)e_1,\;\;[{e_1}_\lambda e_2]=(\partial+\lambda)e_2,\;\;[{e_2}_\lambda e_2]=0,\\
&&\delta(e_1)=\partial e_2\otimes e_2-e_2\otimes \partial e_2,\;\; \delta(e_2)=0.
\end{eqnarray*}
\vspe
\end{ex}

\begin{ex}
Let $(A={\bf k}[x],\circ, \Delta)$ be the Novikov bialgebra given in \mcite{HBG1}, where
\vspb
\begin{eqnarray*}
&&x^m\circ x^n=nx^{m+n-1},\;\;\Delta(1)=\Delta(x)=0,\;\;\Delta(x^n)=-\sum_{i=1}^{n-1}ix^{n-1-i}\otimes x^{i-1}, \;\;\; n\geq 2.
\vspb
\end{eqnarray*}
By Corollary \mref{consthm1}, there is a Lie conformal bialgebra structure on
$R={\bf k}[\partial][x]$ given by
\vspb
\begin{eqnarray*}
&&[{x^m}_\lambda x^n]=(m\partial+(m+n)\lambda)x^{m+n-1},\\
&&\delta (1)=\delta(x)=0,\;\;\delta(x^n)=-\sum_{i=1}^{n-1} i(\partial x^{n-1-i}\otimes x^{i-1}-x^{i-1}\otimes \partial x^{n-1-i}), \;\; n\geq 2.
\end{eqnarray*}
\vspd
\end{ex}

\vspd

\section{Gel'fand-Dorfman Yang-Baxter equation and the relationship with classical conformal Yang-Baxter equation}
\mlabel{s:ybe} In this section, we investigate a special class of
\gdbas which motivates us to introduce the notion of the
Gel'fand-Dorfman Yang-Baxter equation (GDYBE). The operator form
of the GDYBE is also studied. We show that an
$\mathcal{O}$-operator on a \gda associated to a representation
produces a skew-symmetric solution of the GDYBE. It is also shown
that a pre-\gda naturally induces an  $\mathcal{O}$-operator of
the associated \gda and therefore a \gdba. Moreover, the
relationship between GDYBE and classical conformal
Yang-Baxter equation (CCYBE) is investigated. Finally, we
show that there is a natural construction of Lie conformal
bialgebras from pre-\gdas.

\subsection{GDYBE, $\mathcal{O}$-operators and pre-\gdas}
First, we recall some facts about the Novikov Yang-Baxter equation and classical Yang-Baxter equation.

Let $(A, \circ, [\cdot,\cdot])$ be a \gda. Let $r=\sum_\alpha x_\alpha \otimes y_\alpha \in A\otimes A$ and
$r'=\sum_\beta x_\beta '\otimes y_\beta '\in A\otimes A$. Set
$$r_{12}\circ r'_{13}:=\sum_{\alpha,\beta}x_\alpha \circ x_\beta'\otimes y_\alpha\otimes
y_\beta',\;\; r_{12}\circ r'_{23}:=\sum_{\alpha,\beta} x_\alpha\otimes y_\alpha\circ
x_\beta'\otimes y_\beta',\;\;r_{13}\circ r'_{23}:=\sum_{\alpha,\beta} x_\alpha\otimes
x_\beta'\otimes y_\alpha\circ y_\beta',$$
$$r_{13}\circ r'_{12}:=\sum_{\alpha,\beta}x_\alpha\circ x_\beta'\otimes y_\beta'\otimes
y_\alpha,\;\;r_{23}\circ r'_{13}:=\sum_{\alpha,\beta} x_\beta'\otimes x_\alpha\otimes
y_\alpha\circ y_\beta',$$
$$r_{12}\star r'_{23}:=\sum_{\alpha,\beta} x_\alpha\otimes y_\alpha\star
x_\beta'\otimes y_\beta',\;\;r_{13}\star r'_{23}:=\sum_{\alpha,\beta} x_\alpha\otimes
x_\beta'\otimes y_\alpha\star y_\beta',$$
$$[r_{12}, r'_{13}]:=\sum_{\alpha,\beta} [x_\alpha, x_\beta']\otimes y_\alpha\otimes y_\beta',\;[r_{12}, r'_{23}]:=\sum_{\alpha,\beta} x_\alpha\otimes [ y_\alpha,
x_\beta']\otimes y_\beta',\;[r_{13}, r'_{23}]:=\sum_{\alpha,\beta} x_\alpha \otimes  x_\beta'\otimes[ y_\alpha, y_\beta'].$$
The equation
\vspc
\begin{eqnarray}
{\bf N}(r)\coloneqq r_{13}\circ r_{23} +r_{12}\star r_{23}+r_{13}\circ
r_{12}=0
\end{eqnarray}
is called the {\bf Novikov Yang-Baxter equation (NYBE)} in $(A, \circ)$ (see \mcite{HBG}), and the equation
\vspb
\begin{eqnarray}
{\bf C}(r):=  [r_{12}, r_{13}] +[r_{12}, r_{23}]+[r_{13},r_{23}]=0
\end{eqnarray}
is called the {\bf classical Yang-Baxter equation (CYBE)} in $(A, [\cdot,\cdot])$ (see \mcite{CP}).

\begin{pro}\mlabel{cob-Nov}
Let $(A,\circ)$ be a Novikov algebra and $r\in A\otimes A$. Define
\begin{eqnarray}
\Delta_r: A\rightarrow A\otimes A, \quad
\Delta_r(a)\coloneqq -(L_{\circ}(a)\otimes \id+\id\otimes L_{\star}(a))r,\;\;\; a\in A.
\mlabel{coNov}
\end{eqnarray} Then
 $(A,\circ, \Delta_r)$ is a Novikov bialgebra if and
only if the following conditions hold.
\begin{eqnarray}
\mlabel{cob1}&&\;\;\;(\id\otimes (L_{\circ}(b\circ a)+L_{\circ}(a)L_{\circ}(b))+L_{\star}(a)\otimes L_{\star}(b))(r+\tau r)=0,\\
\mlabel{cob2}
&&\;\;\;(L_{\star}(a)\otimes L_{\star}(b)-L_{\star}(b)\otimes L_{\star}(a))(r+\tau r)=0,\\
\mlabel{cob3}
&&\;\;\;\Big(-L_{\star}(b)\otimes R_{\circ}(a)+L_{\star}(a)\otimes R_{\circ}(b)+R_{\circ}(a)\otimes L_{\circ}(b)-R_A(b)\otimes L_{\circ}(a)\\
&&\;\;\;\;\;+\id\otimes \big(L_{\circ}(a)L_{\circ}(b)-L_{\circ}(b)L_\circ(a)\big)-\big(L_\circ(a)L_\circ(b)-L_\circ(b)L_\circ(a)\big)\otimes \id\Big)(r+\tau r)=0,\nonumber\\
&&\mlabel{cob4}\;\;\;\;\Big(L_\circ(a)\otimes \id\otimes \id-\id\otimes L_\circ(a)\otimes \id\Big)\Big((\tau r)_{12}\circ r_{13}+r_{12}\circ r_{23}+r_{13}\star r_{23}\Big)\\
&&\hspace{0.3cm}+\Big((\id\otimes L_\circ(a)\otimes
\id)(r+\tau r)_{12}\Big)\circ r_{23}-\Big((L_\circ(a)\otimes \id\otimes \id)r_{13}\Big)\circ (r+\tau r)_{12}+\Big(\id\otimes \id\otimes
L_{\star}(a)\Big)\nonumber\\
&&\hspace{0.3cm}\Big(r_{23}\circ r_{13}-r_{13}\circ
r_{23}-(\id\otimes\id\otimes \id-\tau\otimes \id)(r_{13}\circ
r_{12}+r_{12}\star r_{23})\Big)=0,\nonumber\\
&&\;\;\;\mlabel{cob5}(\id\otimes \id\otimes \id-\id\otimes \tau)(\id\otimes \id\otimes
            L_{\star}(a))(r_{13}\circ (\tau r)_{23} -r_{12}\star
            r_{23}-r_{13}\circ r_{12})=0,\;\;a, b\in A.
\end{eqnarray}
\end{pro}
\begin{proof}
Note that $(A, \circ, \Delta_r)$ is a Novikov bialgebra if and only if  $(A, \circ, -\Delta_r)$ is a Novikov bialgebra. Then this proposition follows from \cite[Theorem 3.22]{HBG} directly.
\end{proof}

\begin{pro}\mcite{CP}\mlabel{cob-Lie}
Let $(A, [\cdot,\cdot])$ be a Lie algebra and $r\in A\otimes A$. Define
\vspb
\begin{eqnarray}
\mlabel{coLie} \delta_r: A\rightarrow A\otimes A,
\quad \delta_r(a)\coloneqq (\ad(a)\otimes \id+\id\otimes
\ad(a))r,\;\;\; a\in A.
\end{eqnarray} Then
 $(A,[\cdot,\cdot], \delta_r)$ is a Lie bialgebra if and
only if the following conditions hold.
\vspb
\begin{eqnarray}
\mlabel{cob6}&&\Big(\ad(a)\otimes \id+\id\otimes \ad(a)\Big)(r+\tau r)=0,\\
\mlabel{cob7}&&\Big(\ad(a)\otimes \id\otimes \id+\id\otimes \ad(a)\otimes \id+\id\otimes \id\otimes \ad(a)\Big){\bf C}(r)=0, \;\;a\in A.
\end{eqnarray}
Such a Lie bialgebra is called a {\bf coboundary Lie bialgebra}.
\end{pro}

%Next, we give a characterization of a special class of GD-bialgebras.

\begin{thm}\mlabel{thm-cob-GD}
Let $(A, \circ, [\cdot,\cdot])$ be a \gda and $r=\sum_\alpha
x_\alpha\otimes y_\alpha\in A\otimes A$. Let $\Delta_r$ and
$\delta_r$ be the linear maps defined by Eqs. \meqref{coNov} and
\meqref{coLie} respectively. Then $(A, \circ, [\cdot,\cdot],
\Delta_r, \delta_r)$ is a \gdba if and only if Eqs.
\meqref{cob1}-\meqref{cob5}, Eqs. \meqref{cob6}-\meqref{cob7} and
the following equations hold. \vspb
\begin{eqnarray}
&&\mlabel{co-cond1}\Big(L_{\circ}(a)\otimes \id\otimes \id+\id\otimes L_{\star}(a)\otimes \id+\id\otimes \id\otimes L_{\star}(a)\Big){\bf C}(r)\\
&&\qquad-(\id\otimes L_{\star}(a)\otimes \id)[r_{13}, (r+\tau r)_{23}]-\sum_\beta\Big((R_{\circ}(x_\beta)\circ \ad(a))\otimes \id\Big)(r+\tau r)\otimes y_\beta\nonumber\\
\vspb
&&\qquad-(\id\otimes \ad(a)\otimes\id)\big(r_{12}\circ r_{23}+\tau(r)_{12}\circ r_{13}+r_{13}\star r_{23}\big)-(\id\otimes \id\otimes \ad(a)){\bf N}(r)=0,\nonumber\\
\vspb
&&\mlabel{co-cond2}(L_{\star}(a)\otimes \id)(\id\otimes \ad(b))(r+\tau r)+\Big(\id\otimes (\ad(b)\circ L_{\circ}(a))\Big)(r+\tau r)=0, \ a, b\in A.
\end{eqnarray}
\end{thm}

\begin{proof}
By Propositions \mref{cob-Nov} and \mref{cob-Lie}, we only need to
prove that Eqs. (\mref{Lc2}) and (\mref{Lb4}) hold if and only if
Eqs. (\mref{co-cond1}) and (\mref{co-cond2}) hold respectively
when $\Delta=\Delta_r$ and $\delta=\delta_r$. Then for $a\in A$, we have
{\small\begin{eqnarray*}
&&(\id\otimes \delta_r)\Delta_r(a)-(\tau\otimes \id)(\id\otimes\Delta_r)\delta_r(a)+(\tau\otimes \id)(\id\otimes \delta_r)\tau \Delta_r(a)-(\Delta_r\otimes \id)\delta_r(a)-(\delta_r\otimes \id)\Delta_r(a)\\
&&\mbox{}\hspace{0.5cm} = -\sum_{\alpha,\beta}\Big(\big([a\circ x_\alpha,x_\beta]-[a,x_\beta]\circ x_\alpha-[a\circ x_\beta,x_\alpha]\big)\otimes y_\alpha\otimes y_\beta\\
&&\quad\quad\quad+x_\alpha\otimes\big([a\star y_\alpha, x_\beta]-[a,x_\beta]\star y_\alpha-[a\circ x_\beta,y_\alpha]\big)\otimes y_\beta\\
&&\quad\quad\quad+x_\alpha\otimes x_\beta\otimes \big([a\star y_\alpha,y_\beta]-[a,y_\beta]\star y_\alpha\big)+a\circ x_\alpha \otimes [y_\alpha,x_\beta]\otimes y_\beta\\
&&\quad\quad\quad+a\circ x_\alpha\otimes x_\beta\otimes[y_\alpha,y_\beta]-y_\alpha\circ x_\beta\otimes [a,x_\alpha]\otimes y_\beta-x_\beta\otimes [a,x_\alpha]\otimes y_\alpha\star y_\beta\\
&&\quad\quad\quad-[a,y_\alpha]\circ x_\beta\otimes x_\alpha\otimes y_\beta+x_\beta\otimes y_\alpha\otimes [a\circ x_\alpha,y_\beta]+[x_\alpha,x_\beta]\otimes a\star y_\alpha\otimes y_\beta\\
&&\quad\quad\quad+x_\beta\otimes a\star y_\alpha\otimes [x_\alpha,y_\beta]-x_\alpha\circ x_\beta\otimes y_\beta\otimes [a,y_\alpha]-x_\beta\otimes x_\alpha \star y_\beta\otimes [a, y_\alpha]\\
&&\quad\quad\quad-[x_\alpha,x_\beta]\otimes y_\beta\otimes a\star y_\alpha-x_\beta\otimes [x_\alpha,y_\beta]\otimes a\star y_\alpha\Big)\\ %\lir{\text{is the last ) needed?}}\\
&&\mbox{}\hspace{0.5cm} =-\sum_{\alpha,\beta}\Big(\big(a\circ [x_\alpha,x_\beta]-[a,x_\alpha]\circ x_\beta\big)\otimes y_\alpha\otimes y_\beta+x_\alpha\otimes(a\star [y_\alpha,x_\beta]+[y_\alpha\circ x_\beta, a])\otimes y_\beta\\
&&\quad\quad\quad+x_\alpha\otimes x_\beta\otimes \big(a\star [y_\alpha, y_\beta]+[y_\alpha\circ y_\beta,a]+[a\circ y_\beta,y_\alpha]\big)+a\circ x_\alpha \otimes [y_\alpha,x_\beta]\otimes y_\beta\\
&&\quad\quad\quad+a\circ x_\alpha\otimes x_\beta\otimes[y_\alpha,y_\beta]-y_\alpha\circ x_\beta\otimes [a,x_\alpha]\otimes y_\beta-x_\beta\otimes [a,x_\alpha]\otimes y_\alpha\star y_\beta\\
&&\quad\quad\quad-[a,y_\alpha]\circ x_\beta\otimes x_\alpha\otimes y_\beta+x_\beta\otimes y_\alpha\otimes [a\circ x_\alpha,y_\beta]+[x_\alpha,x_\beta]\otimes a\star y_\alpha\otimes y_\beta\\
&&\quad\quad\quad+x_\beta\otimes a\star y_\alpha\otimes [x_\alpha,y_\beta]-x_\alpha\circ x_\beta\otimes y_\beta\otimes [a,y_\alpha]-x_\beta\otimes x_\alpha \star y_\beta\otimes [a, y_\alpha]\\
&&\quad\quad\quad-[x_\alpha,x_\beta]\otimes y_\beta\otimes a\star y_\alpha-x_\beta\otimes [x_\alpha,y_\beta]\otimes a\star y_\alpha \Big)\\%\lir{\text{Is the last ) needed?}}\\
&&\mbox{}\hspace{0.5cm}=-(L_{\circ}(a)\otimes \id\otimes \id+\id\otimes L_{\star}(a)\otimes \id+\id\otimes \id\otimes L_{\star}(a)){\bf C}(r)\\
&&\quad\quad\quad+(\id\otimes L_{\star}(a)\otimes \id)[r_{13}, (r+\tau r)_{23}]+\sum_\beta((R_{\circ}(x_\beta)\circ \ad(a))\otimes \id)(r+\tau r)\otimes y_\beta\nonumber\\
&&\quad\quad\quad+(\id\otimes \ad(a)\otimes\id)(r_{12}\circ r_{23}+\tau(r)_{12}\circ r_{13}+r_{13}\star r_{23})+(\id\otimes \id\otimes \ad(a)){\bf N}(r).
\end{eqnarray*}}
%\lir{Why the outside () is needed?}
Therefore,  Eq. (\mref{Lc2}) holds if and only if Eq. (\mref{co-cond1}) holds.  Similarly, for $a, b\in A$, we have
\begin{eqnarray*}
&&\delta_r(b\circ a)+\Delta_r([a,b])-(R_{\circ}(a)\otimes \id)\delta_r(b)-(L_{\circ}(b)\otimes \id)\delta_r(a)-(\id\otimes L_{ \star}(b))\delta_r(a)\\
&&\quad-a.\Delta_r(b)+(\id\otimes \ad(b))(\Delta_r(a)+\tau\Delta_r(a))\\
&& \mbox{}\hspace{0.5cm}= \sum_\alpha \Big(\big([b\circ a,x_\alpha]-[a,b]\circ x_\alpha-b\circ [a,x_\alpha]-[b,x_\alpha]\circ a+[a,b\circ x_\alpha]\big)\otimes y_\alpha\\
&&\quad\quad\quad+x_\alpha\otimes \big([b\circ a,y_\alpha]-[a,b]\star y_\alpha-b\star [a,y_\alpha]+[a,b\star y_\alpha]-[b,a\star y_\alpha]\big)\\
&&\quad\quad\quad-x_\alpha \circ a\otimes [b,y_\alpha]-a\circ x_\alpha \otimes [b,y_\alpha]-y_\alpha\otimes [b, a\circ x_\alpha]-a\star y_\alpha \otimes [b,x_\alpha]\Big)\\
&&\mbox{}\hspace{0.5cm}=-(L_{\star}(a)\otimes \id)(\id\otimes \ad(b))(r+\tau r)-(\id\otimes (\ad(b)\circ L_{\circ}(a)))(r+\tau r).
\end{eqnarray*}
Therefore, Eq. (\mref{Lb4}) holds if and only if Eq. (\mref{co-cond2}) holds. Then the proof is completed.
\end{proof}

\begin{defi}
Let $(A, \circ, [\cdot,\cdot])$ be a \gda and $r\in A\otimes A$. The equation
\vspb
\begin{eqnarray}
{\bf N}(r)={\bf C}(r)=0
\vspb
\end{eqnarray}
is called the {\bf Gel'fand-Dorfman Yang-Baxter equation} (GDYBE) in $A$.
\end{defi}

\begin{pro}\mlabel{cob-GD}
Let $(A, \circ, [\cdot,\cdot])$ be a \gda and $r\in A\otimes A$ be a skew-symmetric solution of the GDYBE in $A$. Define $\Delta_r$ and $\delta_r$  by Eqs. \meqref{coNov} and \meqref{coLie} respectively. Then $(A, \circ, [\cdot,\cdot], \Delta_r, \delta_r)$ is a \gdba.
\vspb
\end{pro}
\begin{proof}
It follows directly from Theorem \mref{thm-cob-GD}.
\end{proof}

Next, we investigate the operator form of the GDYBE.
We first need to recall the notions of a representation of a Novikov algebra \mcite{O3} and of a \gda~\mcite{WH}.

A {\bf representation} of a Novikov algebra $(A,\circ)$ is a triple $(V, l_A,r_A)$, where $V$ is a vector space and   $l_A$, $r_A: A\rightarrow {\rm End}_{\bf k}(V)$ are linear maps satisfying
    \begin{eqnarray}
        \mlabel{lef-mod1}
        &l_A(a\circ b-b\circ a)v=l_A(a)l_A(b)v-l_A(b)l_A(a)v,&\\
        \mlabel{lef-mod2}
        &l_A(a)r_A(b)v-r_A(b)l_A(a)v=r_A(a\circ b)v-r_A(b)r_A(a)v,&\\
        \mlabel{Nov-mod1}\
        &l_A(a\circ b)v=r_A(b)l_A(a)v,
        &\\
        \mlabel{Nov-mod2}
        &r_A(a)r_A(b)v=r_A(b)r_A(a)v, &
        \;a, b\in A, v\in V.
    \end{eqnarray}

Let $(A, \circ, [\cdot, \cdot])$ be a \gda and $V$ be a vector space. Let $\rho_A$, $l_A$, and $r_A$: $A \rightarrow {\rm
        End}_{\bf k}(V)$ be three linear maps. Then $(V, l_A, r_A, \rho_A)$ is called a {\bf representation} of $(A, \circ, [\cdot, \cdot])$ if $(V,\rho_A)$ is a representation of the Lie algebra $(A, [\cdot,\cdot])$, $(V, l_A, r_A)$ is a representation of the Novikov algebra $(A, \circ)$, and the following conditions are satisfied
    \begin{eqnarray}
        &&\mlabel{GF-rep-1}\;\;\rho_A(a)l_A(b)v+\rho_A(b\circ a)v+l_A([b,a])v-r_A(a)\rho_A(b)v-l_A(b)\rho_A(a) v=0,\\
        &&\mlabel{GF-rep-2}\;\;\rho_A(a)r_A(b)v-\rho_A(b)r_A(a)v-r_A(b)\rho_A(a) v+r_A(a)\rho_A(b) v-r_A([a,b])v =0,\;\;\;a, b\in A,\; v\in V.
    \end{eqnarray}

Note that $(A, L_{\circ}, R_{\circ}, \ad)$ is a representation of a \gda $(A, \circ, [\cdot,\cdot])$, called the {\bf adjoint representation} of $(A, \circ, [\cdot,\cdot])$.
\vspb

\begin{pro}\label{semi-GD}
    Let $(A, \circ, [\cdot,\cdot])$ be a \gda and $V$ be a vector space. Let $l_A$, $r_A$, $\rho_A: A\rightarrow {\rm
        End}_{\bf k}(V)$ be three linear maps. Define binary operations on the direct sum $A\oplus V$ of vector spaces as follows.
\vspb
    \begin{eqnarray}
        &&(a+u)\bullet (b+v):=a\circ b+l_A(a)v+r_A(b)u,\\
        && [a+u, b+v]:=[a,b]+\rho_A(a)v-\rho_A(b)u,\;\;a, b\in A, u, v\in V.
    \end{eqnarray}
    Then $(A\oplus V, \bullet, [\cdot,\cdot])$ is a \gda if and only if $(V, l_A, r_A, \rho_A)$ is a representation of $(A, \circ, [\cdot,\cdot])$. The  \gda $(A\oplus V, \bullet, [\cdot,\cdot])$ is called the {\bf semi-direct product of $(A, \circ, [\cdot,\cdot])$ and its representation $(V, l_A, r_A, \rho_A)$}. We denote it by $A\ltimes_{l_A, r_A, \rho_A} V$.
\end{pro}
\vspb
\begin{proof}
It can be obtained either by a straightforward calculation or as a direct consequence of Proposition~\mref{matched-pair-GD} by taking the trivial operations on $B=V$.
\vspb
\end{proof}

Let $(A, \circ, [\cdot,\cdot])$ be a \gda and $V$ be a vector space. Let $\varphi:A \rightarrow {\rm
    End}_{\bf k}(V)$ be a linear map. Define a linear map $\varphi^\ast:A \rightarrow {\rm
    End}_{\bf k}(V^\ast)$ by
\vspb
\begin{eqnarray*}
    \langle \varphi^\ast(a)f, u\rangle=-\langle f, \varphi(a)u\rangle,\;\;\; a\in A, f\in V^\ast, u\in V.
\end{eqnarray*}
\vspb

\begin{pro}
    Let $(A, \circ, [\cdot,\cdot])$ be a \gda and $(V, l_A, r_A, \rho_A)$ be a representation. Then $(V^\ast, l_A^\ast+r_A^\ast, -r_A^\ast, \rho_A^\ast)$ is a representation of $(A, \circ, [\cdot,\cdot])$.
\end{pro}
\begin{proof}
    By \cite[Proposition 3.3]{HBG}, $(V^\ast, l_A^\ast+r_A^\ast, -r_A^\ast)$ is a representation of the Novikov algebra $(A, \circ)$. Note that $(V^\ast, \rho_A^\ast)$ is a representation of the Lie algebra $(A, [\cdot,\cdot])$. Let $a,b\in A$, $u\in V$ and $f\in V^*$. Therefore we only need to prove
    {\small \begin{eqnarray}
            &&\rho_A^\ast(a)(l_A^\ast+r_A^\ast)(b)f+\rho_A^\ast(b\circ a)f+(l_A^\ast+r_A^\ast)([b,a]) f+r_A^\ast(a)\rho_A^\ast(b)f-(l_A^\ast+r_A^\ast)(b)\rho_A^\ast(a)f=0,\mlabel{eq:rr1}\\
            &&-\rho_A^\ast(a)r_A^\ast(b)f+\rho_A^\ast(b)r_A^\ast(a)f+r_A^\ast(b)\rho_A^\ast(a)f-r_A^\ast(a)\rho_A^\ast(b)f+r_A^\ast([a,b])f=0. \mlabel{eq:rr2}
    \end{eqnarray}}
In fact, adding Eqs. (\mref{GF-rep-1}) and (\mref{GF-rep-2}) up, we obtain
\vspb
\begin{eqnarray*}
(l_A+r_A)(b)\rho_A(a)u-\rho_A(b\circ a)u-(l_A+r_A)([b,a])u+\rho_A(b)r_A(a)u-\rho_A(a)(l_A+r_A)(b)u=0.
\end{eqnarray*}
Then we have
\vspb
{\small \begin{eqnarray*}
&&\langle \rho_A^\ast(a)(l_A^\ast+r_A^\ast)(b)f+\rho_A^\ast(b\circ a)f+(l_A^\ast+r_A^\ast)([b,a]) f+r_A^\ast(a)\rho_A^\ast(b)f-(l_A^\ast+r_A^\ast)(b)\rho_A^\ast(a)f, u\rangle\\
&&\mbox{}\hspace{0.5cm}=\langle f, (l_A+r_A)(b)\rho_A(a)u-\rho_A(b\circ a)u-(l_A+r_A)([b,a])u+\rho_A(b)r_A(a)u-\rho_A(a)(l_A+r_A)(b)u\rangle\\
&&\mbox{}\hspace{0.5cm}= 0.
\end{eqnarray*}}
    Therefore Eq.~(\mref{eq:rr1}) holds.   Similarly,  Eq.~(\mref{eq:rr2}) holds.
\end{proof}

\begin{rmk}
    The adjoint representation $(A, L_{\circ}, R_{\circ}, \ad)$ of a \gda $(A, \circ,$ $ [\cdot,\cdot])$ gives the representation $(A^\ast, L_{ \star}^\ast, -R_{
        \circ}^\ast, \ad^\ast)$.
\end{rmk}

For a finite-dimensional vector space $A$, the isomorphism
\vspb
$$A\ot A\cong
\Hom_{\bf k}(A^*,{\bf k})\ot A\cong \Hom_{\bf k}(A^*,A)
\vspa
$$
identifies an $r\in A\otimes
A$ with a map from $A^*$ to $A$ which we denote by $T^r$.
Explicitly, writing $r=\sum_{\alpha}x_\alpha\otimes y_\alpha$, then
\vspb
\begin{equation}\mlabel{eq:4.12} \notag
T^r:A^*\to A, \quad T^r(f)=\sum_{\alpha}\langle f, x_\alpha\rangle y_\alpha,\;\; f\in A^*.
\vspb
\end{equation}

\begin{pro}\mlabel{oper-form1}
Let $(A, \circ, [\cdot,\cdot])$ be a \gda and $r\in A\otimes A$ be skew-symmetric. Then $r$ is a solution of the GDYBE in $(A, \circ, [\cdot,\cdot])$ if and only if $T^r$ satisfies
\begin{eqnarray}
&&\mlabel{operr1}T^r(f)\circ
T^r(g)=T^r(L_{\star}^\ast(T^r(f))g)-T^r(R_{\circ}^\ast(T^r(g))f),\\
&&\mlabel{operr2}[T^r(f), T^r(g)]=T^r(\ad^\ast(T^r(f))g-\ad^\ast(T^r(g))f, \;\;\;f,~~g\in
A^\ast.
\end{eqnarray}
\end{pro}
\begin{proof}
Note that $r$ is a solution of the GDYBE in $(A, \circ, [\cdot,\cdot])$ if and only if  $r$ is both a solution of the NYBE in $(A, \circ)$ and a solution of the CYBE in $(A, [\cdot,\cdot])$. It is well known that $r$ is a solution of the CYBE in $(A, [\cdot,\cdot])$ if and only if Eq. (\mref{operr2}) holds (see \mcite{Ku}). By \cite[Theorem 3.27]{HBG}, $r$ is a solution of the NYBE in $(A, \circ)$ if and only if Eq. (\mref{operr1}) holds. Then the conclusion follows.
\vspb
\end{proof}

Proposition \mref{oper-form1} motivates us to give the following definition.

\begin{defi}
Let $(A, \circ, [\cdot,\cdot])$ be a \gda and $(V, l_A, r_A, \rho_A)$ be a representation of $A$. A linear map
$T: V\rightarrow A$ is called an {\bf $\mathcal{O}$-operator} on $(A, \circ, [\cdot,\cdot])$ associated to $(V, l_A, r_A, \rho_A)$ if $T$ satisfies
\vspc
\begin{eqnarray}
&&\mlabel{operr3}T(u)\circ T(v)=T(l_A(T(u))v+r_A(T(v))u),\\
&&\mlabel{operr4}[T(u),T(v)]=T(\rho_A(T(u))v-\rho_A(T(v))u),\;\;u, v\in V.
\end{eqnarray}
\vspb
\end{defi}
Eq. (\mref{operr3}) means that $T$ is an $\mathcal{O}$-operator on the Novikov algebra $(A, \circ)$ associated to $(V, l_A, r_A)$ (see \mcite{HBG}); while Eq. (\mref{operr4}) means that $T$ is an $\mathcal{O}$-operator on the Lie algebra $(A, [\cdot,\cdot])$ associated to $(V, \rho_A)$ (see \mcite{Ku}). On the other hand, by Proposition \mref{oper-form1}, if $r$ is skew-symmetric, then $r$ is a solution of the GDYBE in the \gda $(A, \circ, [\cdot,\cdot])$ if and only if $T^r$ is an $\mathcal{O}$-operator on $(A, \circ, [\cdot,\cdot])$ associated to $(A^\ast, L_{\star}^\ast, -R_{\circ}^\ast, \ad^\ast)$.
We give the following construction of solutions of the GDYBE from $\calo$-operators on \gdas.
\begin{thm}\mlabel{oper-form2}
Let $(A, \circ, [\cdot,\cdot])$ be a \gda and $(V, l_A, r_A, \rho_A)$ be a representation. Let $T: V\rightarrow A$ be a linear map which is
identified with $r_T\in A\otimes V^\ast \subseteq
(A\ltimes_{l_A^\ast+r_A^\ast,-r_A^\ast, \rho_A^\ast} V^\ast) \otimes
(A\ltimes_{l_A^\ast+r_A^\ast,-r_A^\ast, \rho_A^\ast} V^\ast)$ through
 ${\rm Hom}_{\bf k}(V, A)\cong A\otimes V^\ast$.
Then $r=r_T-\tau r_T$ is a solution of the GDYBE in the \gda
$(A\ltimes_{l_A^\ast+r_A^\ast,-r_A^\ast, \rho_A^\ast} V^\ast, \bullet, [\cdot,\cdot])$
if and only if $T$ is an
$\mathcal{O}$-operator on $(A,\circ, [\cdot,\cdot])$ associated to $(V, l_A,
r_A, \rho_A)$.
\end{thm}
\begin{proof}
It follows directly by combining \cite[Theorem 3.29]{HBG} and \cite[Section 2]{Bai}.
\end{proof}

\begin{defi}\mlabel{ND} \mcite{HBG}
Let $A$ be a vector space with binary operations $ \lhd$ and $\rhd$. If for all $a$, $b$ and $c\in A$, they satisfy the following equalities
\vspb
\begin{align}
&a\rhd(b\rhd c)=(a\rhd b+a\lhd b)\rhd c+b\rhd(a\rhd c)-(b\rhd a+b\lhd a)\rhd c,\mlabel{ND1}\\
&a\rhd(b\lhd c)=(a\rhd b)\lhd c+b\lhd(a\lhd c+a\rhd c)-(b\lhd a)\lhd c,\mlabel{ND2}\\
&(a\lhd b+a\rhd b)\rhd c=(a\rhd c)\lhd b,\mlabel{ND3}\\
&(a\lhd b)\lhd c=(a\lhd c)\lhd b,\mlabel{ND4}
\end{align}
then $(A,\lhd,\rhd)$ is called a \bf{pre-Novikov algebra}.
\vspb
\end{defi}

\begin{defi}\mcite{XH}
Let $(A, \lhd, \rhd)$ be a pre-Novikov algebra and $(A, \diamond)$ be a left-symmetric algebra. If they satisfy the following compatibility conditions:
\vspb
\begin{eqnarray}
&&\;\;\;c\lhd(a\diamond b-b\diamond a)-a\diamond(c\lhd b)-(b\diamond c)\lhd a=-b\diamond(c\lhd a)-(a\diamond c)\lhd b,\mlabel{lnd1}\\
&&\;\;\;(a\diamond b-b\diamond a)\rhd c+(a\lhd b+a\rhd b)\diamond c
=a\rhd(b\diamond c)-b\diamond(a\rhd c)+(a\diamond c)\lhd b, a, b, c\in A,\mlabel{lnd2}
\end{eqnarray}
then the quadruple
$(A,\lhd,\rhd,\diamond)$ is called a {\bf pre-Gel'fand-Dorfman algebra (pre-\gda)}.
\vspb
\end{defi}
\begin{ex}\cite[Corollary 2.24]{XH}\mlabel{constr-pre-GD}
Recall \mcite{Lo} that a {\bf Zinbiel algebra} $(A, \cdot)$ is a vector space $A$ with a binary operation $\cdot: A\otimes A\rightarrow A$ satisfying
\vspb
\begin{eqnarray*}
a\cdot (b\cdot c)=(b\cdot a+a\cdot b)\cdot c,\;\;\;a, b, c\in A.
\end{eqnarray*}
Let $D$ be a derivation on $(A,\cdot)$ and $\xi$, $k\in {\bf k}$. Define
\vspb
\begin{eqnarray*}
&&a\lhd b:=D(b)\cdot a+\xi  b\cdot a ,~~a\rhd b:=a\cdot D(b)+\xi a \cdot b,\\
&&a\diamond b:=k(a\cdot D(b)- D(a)\cdot b+\xi (a\cdot b-b\cdot a)), \;\;\; a, b\in A.
\end{eqnarray*}
Then $(A,\lhd,\rhd, \diamond)$ is a pre-\gda.
\end{ex}

The relationship between \gdas and pre-\gdas is given as follows.
\begin{pro}\mlabel{GD-PGD}\cite[Proposition 2.19]{XH}
Let $(A,\lhd,\rhd, \diamond)$ be a pre-\gda. Define
\vspb
\begin{eqnarray}\mlabel{GD1}
a\circ b=a\lhd b+a\rhd b,~~[a ,b]=a\diamond b-b\diamond a,\;\;\;a, b\in A.
\end{eqnarray}
Then $(A,\circ,[\cdot,\cdot])$ is a \gda, called the {\bf associated \gda} of $(A, \lhd, \rhd, \diamond)$.
Moreover,  $(A,  L_{\rhd}, R_{\lhd}, L_\diamond)$ is
a representation of the \gda $(A, \circ, [\cdot, \cdot])$.
Conversely, let $A$ be a vector space
with binary operations $\rhd$, $\lhd$ and $\diamond$. If $(A, \circ, [\cdot, \cdot])$ defined by Eq. \meqref{GD1} is a \gda and $(A,  L_{\rhd}, R_{\lhd}, L_\diamond)$ is a representation of $(A, \circ, [\cdot, \cdot])$, then $(A, \lhd, \rhd, \diamond)$ is a pre-\gda.
\end{pro}

The close relationship among pre-GD algebras, $\calo$-operators on \gdas and solutions of the GDYBE can be formulated as follows.

\begin{thm} \mlabel{NYB-ND}
\begin{enumerate}
\item \mlabel{it:1}Let $(A, \circ, [\cdot,\cdot])$ be a \gda, $(V, l_A,
r_A, \rho_A)$ be a representation of $(A, \circ, [\cdot,\cdot])$ and $T: V\rightarrow A$ be an
$\mathcal{O}$-operator associated to $(V, l_A, r_A, \rho_A)$. Then there
exists a pre-\gda structure on $V$ defined by
\vspb
\begin{eqnarray}\mlabel{eq:ndend} \notag
u\rhd v:=l_A(T(u))v,\;\; u\lhd v:= r_A(T(v))u, \;\; u\diamond v:=\rho_A(T(u))v,\;\;\; u,~~v\in V.
\end{eqnarray}
\item\mlabel{it:2}
Let $(A, \lhd, \rhd, \diamond)$ be a pre-\gda and $(A, \circ, [\cdot,\cdot])$ be the
associated \gda. Then the identity map is an $\mathcal
O$-operator on $(A, \circ, [\cdot,\cdot])$ associated to the representation $(A,
L_{\rhd}, R_{\lhd}, L_\diamond)$.
\item \mlabel{it:3}
Let $(A, \lhd, \rhd, \diamond)$ be a pre-\gda and $(A,
\circ, [\cdot,\cdot])$ be the associated \gda.
Let $\{e_1, \ldots, e_n\}$ be a linear basis of
$A$ and $\{e_1^\ast, \ldots, e_n^\ast\}$ be the dual basis of
$A^\ast$.
Then
\vspb
\begin{eqnarray}\mlabel{eq:solu}
r:=\sum_{\alpha=1}^n(e_\alpha\otimes e_\alpha^\ast-e_\alpha^\ast\otimes e_\alpha),
\vspb
\end{eqnarray}
is a skew-symmetric solution of the GDYBE in
the \gda $A\ltimes_{L_\rhd^\ast+R_\lhd^\ast,
-R_\lhd^\ast, L_\diamond^\ast}A^\ast$.
\end{enumerate}
\end{thm}
\begin{proof}
\meqref{it:1} and \meqref{it:2} follow from direct verification.
For \meqref{it:3}, since the identity map
$\id: A\rightarrow A$ is an $\mathcal{O}$-operator on $(A, \circ, [\cdot,\cdot])$
associated to $(A, L_\rhd, R_\lhd, L_\diamond)$, the
conclusion follows from Theorem \mref{oper-form2}.
\end{proof}

\subsection{GDYBE and classical conformal Yang-Baxter equation}
We first recall the conformal classical Yang-Baxter equation \mcite{L} which appears in the coboundary case of Lie conformal bialgebras.

Let $(R, [\cdot_\lambda \cdot])$ be a Lie conformal algebra and $r=\sum_\alpha x_\alpha\otimes y_\alpha\in R\otimes R$. Let $\partial^{\otimes^3}:=\partial\otimes \id\otimes \id+\id\otimes \partial\otimes \id+\id\otimes \id\otimes \partial$.
The following equation
\begin{eqnarray*}
[[r,r]]:&=& \sum_{\alpha,\beta}([{x_\alpha}_\mu x_\beta]\otimes y_\alpha\otimes y_\beta|_{\mu=\id\otimes \partial\otimes \id}\\
&&-x_\alpha\otimes [{x_\beta}_\mu y_\alpha]\otimes y_\beta|_{\mu=\id\otimes \id\otimes \partial}-x_\alpha\otimes x_\beta\otimes [{y_\beta}_\mu y_\alpha]|_{\mu=\id\otimes \partial\otimes \id})\\
&\equiv &0 \;\; \text{mod}~(\partial^{\otimes^3}) ~~~\text{in $R\otimes R\otimes R$}
\end{eqnarray*}
is called the {\bf conformal classical Yang-Baxter equation in $R$} or simply {\bf CCYBE in $R$}.

\begin{pro}\cite[Theorem 3.4]{L}\mlabel{coboundary-conf}
Let $R$ be a Lie conformal algebra. If $r=\sum_\alpha x_\alpha\otimes y_\alpha \in R\otimes R$ is a skew-symmetric solution of the CCYBE in $R$, then $(R, [\cdot_\lambda \cdot], \delta)$ is a Lie conformal bialgebra, where $\delta$ is defined by
\vspb
\begin{eqnarray}\mlabel{coboundary}
\delta(a)=a_\lambda
r|_{\lambda=-\partial^{\otimes^2}}=\sum_\alpha ([a_\lambda x_\alpha]\otimes y_\alpha+x_\alpha\otimes [a_\lambda y_\alpha])|_{\lambda=-\partial^{\otimes^2}},\;\;a\in R.
\vspb
\end{eqnarray}
\end{pro}

The relationship between GDYBE and CCYBE is given as follows.
\begin{pro}\mlabel{conformal Yang-Baxter equation}
Let $(A, \circ, [\cdot,\cdot])$ be a \gda, and let $R={\bf k}[\partial]A$ be the Lie conformal algebra corresponding to $(A, \circ, [\cdot,\cdot])$. Let $r
\in A\otimes A$ be skew-symmetric and also naturally regarded as an element in $R\ot R$. Then
$r\in R\otimes R$ is a solution of the CCYBE in
$R$ if and only if $r$ is a solution of the GDYBE in $A$.
\vspb
\end{pro}

\begin{proof} Set $r=\sum_\alpha x_\alpha\otimes y_\alpha\in A\otimes A$.
Note that {\small \begin{eqnarray*}
[[r,r]]&=&\sum_{\alpha,\beta}\Big(\big(\partial(x_\beta\circ x_\alpha)+\mu({x_\alpha}\star x_\beta)+[x_\alpha, x_\beta]\big)\otimes y_\alpha\otimes y_\beta|_{\mu=\id\otimes \partial\otimes \id}\\
&&-x_\alpha\otimes \big(\partial(y_\alpha\circ x_\beta)+\mu (x_\beta\star y_\alpha)+[x_\beta, y_\alpha]\big)\otimes y_\beta|_{\mu=\id\otimes \id\otimes \partial}\\
&&-x_\alpha\otimes x_\beta\otimes \big(\partial (y_\alpha\circ y_\beta)+\mu(y_\alpha\star y_\beta)+[y_\beta, y_\alpha]\big)|_{\mu=\id\otimes \partial\otimes \id}\Big)\\
&=&\sum_{\alpha,\beta}\Big(\partial(x_\beta\circ x_\alpha)\otimes y_\alpha\otimes y_\beta+{x_\alpha}\star x_\beta \otimes \partial y_\alpha\otimes y_\beta+[x_\alpha, x_\beta]\otimes y_\alpha\otimes y_\beta\\
&&-x_\alpha\otimes \partial(y_\alpha\circ x_\beta)\otimes y_\beta-x_\alpha\otimes x_\beta \star y_\alpha\otimes \partial y_\beta-x_\alpha\otimes [x_\beta, y_\alpha]\otimes y_\beta\\
&&-x_\alpha\otimes x_\beta\otimes \partial (y_\alpha\circ y_\beta)-x_\alpha\otimes \partial x_\beta\otimes y_\alpha\star y_\beta-x_\alpha\otimes x_\beta\otimes [y_\beta, y_\alpha]\Big)\\
&\equiv &\sum_{\alpha,\beta}\Big((-\id\otimes \partial\otimes \id-\id\otimes \id\otimes \partial)(x_\beta\circ x_\alpha)\otimes y_\alpha\otimes y_\beta+{x_\alpha}\star x_\beta\otimes \partial y_\alpha\otimes y_\beta-x_\alpha\otimes \partial(y_\alpha\circ x_\beta)\otimes y_\beta\\
&&-x_\alpha\otimes x_\beta \star y_\alpha\otimes \partial y_\beta-x_\alpha\otimes x_\beta\otimes \partial (y_\alpha\circ y_\beta)-x_\alpha\otimes \partial x_\beta\otimes y_\alpha\star y_\beta+{\bf C}(r)\Big)\;\; \text{mod $(\partial^{\otimes^3})$}\\
&\equiv& \sum_{\alpha,\beta}\Big((\id\otimes \id\otimes \partial)\big(-x_\beta\circ x_\alpha\otimes y_\alpha\otimes y_\beta-x_\alpha\otimes x_\beta\star y_\alpha\otimes y_\beta-x_\alpha\otimes x_\beta\otimes y_\alpha\circ y_\beta\big)\\
&&+(\id\otimes \partial \otimes \id)\big(x_\alpha\circ x_\beta\otimes y_\alpha\otimes y_\beta-x_\alpha\otimes y_\alpha\circ x_\beta\otimes y_\beta-x_\alpha\otimes x_\beta\otimes y_\alpha\star y_\beta\big)+{\bf C}(r)\Big)\;\; \text{mod $(\partial^{\otimes^3})$}.
\end{eqnarray*}}
Then it is straightforward to show that $r$ is a solution of the CCYBE in $R$ if and only if $r$ is a solution of the GDYBE in $A$.
\end{proof}

Let $r\in A\otimes A$ be a skew-symmetric solution of the
GDYBE in a \gda $(A, \circ, [\cdot,\cdot])$. Then by Proposition~\mref{cob-GD},  $(A, \circ, [\cdot,\cdot], \Delta=\Delta_r, \delta_0=\delta_r)$
with $\Delta_r$ and $\delta_r$ defined by Eqs. (\mref{coNov}) and (\mref{coLie}) respectively is a \gdba.
Let $R={\bf k}[\partial]A$
be the Lie conformal algebra corresponding to $(A, \circ, [\cdot,\cdot])$. On the one hand, by Theorem~\mref{thm1-corr}, there is the corresponding Lie conformal bialgebra
$(R, [\cdot_\lambda \cdot], \delta)$ with $\delta$ defined by Eq.~(\mref{cobracket}). On the other hand, by Proposition~\mref{conformal Yang-Baxter equation},
$r$ is a skew-symmetric solution of the CCYBE in $R$ and hence by Proposition~\mref{coboundary-conf}, there is a Lie conformal bialgebra $(R, [\cdot_\lambda \cdot], \bar\delta)$ with $\bar\delta$ defined by Eq.~(\mref{coboundary}).

\begin{cor}\mlabel{cor:r} With the above notations, the two Lie conformal bialgebras $(R, [\cdot_\lambda \cdot], \delta)$ and $(R, [\cdot_\lambda \cdot], \bar\delta)$  coincide. Moreover, we have the following commutative diagram.
\vspc
$$  \xymatrix{
\txt{\tiny $r$\\ \tiny a skewsymmetric solution\\ \tiny of the GDYBE in $(A, \circ, [\cdot,\cdot])$} \ar[rr]^-{\rm Prop. \mref{cob-GD}}   \ar@{<->}[d]_-{\rm Prop. \mref{conformal Yang-Baxter equation}}&& \txt{\tiny $(A, \circ, [\cdot,\cdot], \Delta, \delta_0)$\\ \tiny a \gdba} \ar@{<->}[d]_-{\rm Thm. \mref{thm1-corr}}\\
\txt{\tiny $r$\\ \tiny a skewsymmetric solution of\\ \tiny  the CCYBE in ${\bf k}[\partial]A$} \ar[rr]^-{\rm
Prop. \mref{coboundary-conf}}
            && \txt{\tiny $({\bf k}[\partial]A, [\cdot_\lambda \cdot], \delta)$\\ \tiny  a Lie conformal bialgebra  }
}
\vspb
$$
\end{cor}
\begin{proof}
Set $r=\sum_\alpha x_\alpha\otimes y_\alpha\in A\otimes A$. Let $a\in A$.  Then
\vspb
{\small \begin{eqnarray*}
\delta(a)&=&(\partial\otimes \id)\Delta(a)-\tau (\partial\otimes \id)\Delta(a)+\delta_0(a)\\
&=&-(\partial\otimes \id)\sum_\alpha (a\circ x_\alpha\otimes y_\alpha+x_\alpha\otimes a\star y_\alpha)+\tau(\partial\otimes \id)\sum_\alpha (a\circ x_\alpha\otimes y_\alpha+x_\alpha\otimes a\star y_\alpha)\\
&&+\sum_\alpha ([a, x_\alpha]\otimes y_\alpha+x_\alpha\otimes [a, y_\alpha])\\
&=&-\sum_\alpha(\partial(a\circ x_\alpha)\otimes y_\alpha+\partial x_\alpha\otimes a\star y_\alpha-y_\alpha\otimes \partial(a\circ x_\alpha)
-a\star y_\alpha \otimes \partial x_\alpha\\
&&- [a, x_\alpha]\otimes y_\alpha-x_\alpha\otimes [a, y_\alpha])\\
&=&-\sum_\alpha(\partial(a\circ x_\alpha)\otimes y_\alpha+\partial x_\alpha\otimes a\star y_\alpha+x_\alpha\otimes \partial(a\circ y_\alpha)
+a\star x_\alpha \otimes \partial y_\alpha\\
&&\quad -[a, x_\alpha]\otimes y_\alpha-x_\alpha\otimes [a, y_\alpha])\\
&=&\sum_\alpha((\partial(x_\alpha\circ a)+\lambda(a\star x_\alpha)+[a,x_\alpha])\otimes y_\alpha+x_\alpha\otimes (\partial (y_\alpha\circ a)+\lambda (a\star y_\alpha)+[a,y_\alpha]))|_{\lambda=-\partial^{\otimes^2}}\\
&=&a_\lambda r |_{\lambda=-\partial^{\otimes^2}}=\bar \delta(a).
\end{eqnarray*}
}
The commutative diagram follows directly.
\vspb
\end{proof}

\begin{defi}\mcite{K1}
    Let $M$ and $N$ be ${\bf k}[\partial]$-modules. A {\bf conformal linear map} from $M$ to $N$ is a ${\bf k}$-linear map $f: M\rightarrow N[\lambda]$, denoted by $f_\lambda: M\rightarrow N$, such that $[\partial, f_\lambda]=-\lambda f_\lambda$. Denote the ${\bf k}$-vector space of all such maps by $\text{Chom}(M,N)$.
    %\lir{Is this standard? Use $\text{CHom}$?}
    It has a canonical structure of a ${\bf k}[\partial]$-module
\vspb
    $$(\partial f)_\lambda =-\lambda f_\lambda.$$ Define the {\bf conformal dual} of a ${\bf k}[\partial]$-module $M$ by $M^{\ast c}=\text{Chom}(M,{\bf k})$, where ${\bf k}$ is viewed as the trivial ${\bf k}[\partial]$-module, that is
\vspb
    $$M^{\ast c}=\big\{f:M\rightarrow {\bf k}[\lambda]~~\big|~~\text{$f$ is ${\bf k}$-linear and}~~f_\lambda(\partial b)=\lambda f_\lambda b\big\}.$$
\vspb
\end{defi}

Let $M$ be a finitely generated ${\bf k}[\partial]$-module. Set $\text{Cend}(M)=\text{Chom}(M,M)$.
Then there is a Lie conformal algebra structure on $\text{Cend}(M)$ defined by
\vspb
\begin{eqnarray}
    [f_\lambda g]_\mu v=f_\lambda (g_{\mu-\lambda} v)-g_{\mu-\lambda}(f_{\lambda} v), ~~~ f,~g\in \text{Cend}(M), v\in M.\end{eqnarray}
Denote this Lie conformal algebra by $\text{gc}(M)$, which is called the {\bf general Lie conformal algebra} of $M$.

\begin{defi}\mcite{K1}
    Let $M$ be a finitely generated ${\bf k}[\partial]$-module and $R$ be a Lie conformal algebra. $(M, \rho)$ is called {\bf a representation} of $R$ if $\rho: R\rightarrow \text{gc}(M)$ is a homomorphism of Lie conformal algebras.
\vspb
\end{defi}

\begin{rmk}
We denote the {\bf adjoint representation} of $R$ by $(R, \ad_R)$, where $\ad_R(a)_\lambda b=[a_\lambda b]$, where $a$, $b\in R$. Let $(M, \rho)$ be a representation of $R$. There is also a dual representation $(M^{\ast c}, \rho^\ast)$ given as follows
\vspb
\begin{eqnarray*}
            (\rho^\ast(a)_\lambda \varphi)_\mu u=-\varphi_{\mu-\lambda}(\rho(a)_\lambda u),\;\;\;a\in R, \varphi\in M^{\ast c}, u\in M.
\vspd
    \end{eqnarray*}
\end{rmk}
\begin{pro}\cite[Proposition 3.1]{HL}\label{semi-Lie}
Let $R$ be a Lie conformal algebra, $M$ be a ${\bf
k}[\partial]$-module and $\rho: R\rightarrow \text{gc}(M)$ be a
${\bf k}[\partial]$-module homomorphism. Define a
$\lambda$-bracket on the direct sum $R\oplus M$ as ${\bf
k}[\partial]$-modules as follows.\vspb
\begin{eqnarray}
[(a+u)_\lambda (b+v)]:=[a_\lambda b]+\rho(a)_\lambda v-\rho(b)_{-\lambda-\partial}u,\;\;a, b\in R, \; u, v\in M.
\end{eqnarray}
Then $R\oplus M$ is a Lie conformal algebra if and only if $(M,\rho)$ is a representation of $R$. The Lie conformal algebra $R\oplus M$ is called the {\bf semi-direct product of $R$ and its representation $(M,\rho)$}. We denote it by $R\ltimes_{\rho}M$.
\end{pro}
We next present the relationship between representations of \gdas and a class of representations of the corresponding Lie conformal algebras.
\begin{pro}\label{corr-resp}
    Let $R={\bf k}[\partial]A$ be the Lie conformal algebra corresponding to a \gda $(A, \circ, [\cdot,\cdot])$ and $M={\bf k}[\partial]V$ which is free as a ${\bf k}[\partial]$-module on $V$. Define
    a ${\bf k}[\partial]$-module homomorphism $\rho_{l_A, r_A, \rho_A}: R\rightarrow \text{gc}(M)$ by
\vspb
    \begin{eqnarray}\mlabel{Conformal module}
        {\rho_{l_A, r_A, \rho_A}(a)}_\lambda v=\partial(r_A(a)v)+\lambda(l_A(a)v+r_A(a)v)+\rho_A(a)v,\;\;\; a\in A, v\in V,
    \end{eqnarray}
    where $l_A$, $r_A$ and $\rho_A: A\rightarrow \text{End}_{{\bf k}}(V)$ are linear maps. Then
    $(M, \rho_{l_A,r_A, \rho_A})$ is a representation of $R$ if and only if $(V, l_A, r_A, \rho_A)$ is a representation of $(A, \circ, [\cdot,\cdot])$.
\end{pro}
\begin{proof}
    Let $a$, $b\in A$ and $u$, $v\in V$. Note that
\vspb
    \begin{eqnarray*}
    [(a+u)_\lambda (b+v)]:&=&[a_\lambda b]+\rho_{l_A,r_A,\rho_A}(a)_\lambda v-\rho_{l_A,r_A,\rho_A}(b)_{-\lambda-\partial} u\\
    &=&\partial (b\circ a)+\lambda (a\star b)+[a,b]+\partial(r_A(a)v)+\lambda(l_A(a)v+r_A(a)v)+\rho_A(a)v\\
    &&\quad-(\partial(r_A(b)u)+(-\lambda-\partial)(l_A(b)u+r_A(b)u)+\rho_A(b)u)\\
    &=& \partial(b\circ a+l_A(b)u+r_A(a)v)+\lambda(a\star b+l_A(a)v+r_A(a)v+l_A(b)u+r_A(b)u)\\
    &&\quad+([a,b]+\rho_A(a)v-\rho_A(b)u).
\vspd
    \end{eqnarray*}
    Define
\vspb
$$
(a+u)\bullet (b+v):=b\circ a+l_A(a)v+r_A(b)v,\
[a+u,b+v]:=[a,b]+\rho_A(a)v-\rho_A(b)u,\;\; a, b\in A, \;u, v\in V.
$$
By Proposition \ref{semi-Lie}, $(M, \rho_{l_A,r_A, \rho_A})$ is a representation of $R$ if and only if $R\oplus M$ is a Lie conformal algebra with the  $\lambda$-bracket above. Note that $R\oplus M={\bf k}[\partial](A\oplus V)$ is a Lie conformal algebra if and only if $(A\oplus V, \bullet, [\cdot,\cdot])$ is a \gda by  Proposition \ref{Correspond-Algebra}. Then this conclusion follows by Proposition \ref{semi-GD}.
\vspd
\end{proof}

\begin{rmk}
    Let $R={\bf k}[\partial]A$ be the Lie conformal algebra corresponding to a \gda $(A, \circ, [\cdot,\cdot])$. Then
    the adjoint representation $(R, \ad_R)$ of $R$ is exactly $(R, \rho_{L_{\circ}, R_{\circ}, \ad})$.

    By the proof of Proposition \ref{corr-resp},  $R\ltimes_{\rho_{l_A, r_A,\rho_A}}M$ is just the Lie conformal algebra corresponding to the \gda $A\ltimes_{l_A,r_A,\rho_A} V$. %\lir{Something is missing here.}
\end{rmk}

\begin{pro}\mlabel{dual resp}
    Let $(A, \circ, [\cdot,\cdot])$ be a \gda and $R={\bf k}[\partial]A$ be the Lie conformal algebra corresponding to $(A, \circ, [\cdot,\cdot])$. Suppose that $(V, l_A, r_A, \rho_A)$ is a representation of $(A, \circ, [\cdot,\cdot])$ and $M={\bf k}[\partial]V$  which is free as a ${\bf k}[\partial]$-module on $V$. Then the representation $(M^{\ast c}, \rho_{l_A,r_A,\rho_A}^\ast)$ is just the representation $({\bf k}[\partial]V^\ast, \rho_{l_A^\ast+r_A^\ast, -r_A^\ast, \rho_A^\ast})$.
\end{pro}
\begin{proof}
    Note that $M^{\ast c}={\bf k}[\partial]V^\ast$. Moreover, for all $a\in A$, $f\in V^\ast$ and $v\in V$, we have
\vspb
    \begin{eqnarray*}
        ({\rho_{l_A,r_A,\rho_A}^\ast(a)}_\lambda f)_\mu v&=&-f_{\mu-\lambda}(\rho_{l_A,r_A, \rho_A}(a)_\lambda v)\\
        &=&-f_{\mu-\lambda}(\partial(r_A(a)v)+\lambda(l_A(a)v+r_A(a)v)+\rho_A(a)v)\\
        &=&-(\mu-\lambda)f(r_A(a)v)-\lambda f(l_A(a)v+r_A(a)v)-f(\rho_A(a)v)\\
        &=&-\mu f(r_A(a)v)-\lambda f(l_A(a)v)-f(\rho_A(a)v)\\
        &=& \mu(r_A^\ast(a)f)v+\lambda (l_A^\ast (a)f)v+(\rho_A^\ast(a)f)v\\
        &=&(\partial (-r_A^\ast(a)f)+\lambda((l_A^\ast+r_A^\ast)(a)f-r_A^\ast(a)f)+\rho_A^\ast(a)f)_\mu v\\
        &=&(\rho_{l_A^\ast+r_A^\ast,-r_A^\ast, \rho_A^\ast}(a)_\lambda f)_\mu v.
    \end{eqnarray*}
    Therefore, $\rho_{l_A,r_A, \rho_A}^\ast=\rho_{l_A^\ast+r_A^\ast,-r_A^\ast, \rho_A^\ast}$.
\end{proof}
\vspb
\begin{defi} \mcite{HB} Let $R$ be a Lie conformal algebra and $(M, \rho)$ be a representation of $R$.
If a ${\bf k}[\partial]$-module homomorphism $T: M\rightarrow R$ satisfies
\vspb
\begin{eqnarray*}
[T(u)_\lambda T(v)]=T(\rho(T(u))_\lambda v-\rho(T(v))_{-\lambda-\partial} u),\;\;\; u, v\in M,
\end{eqnarray*}
then $T$ is called an {\bf $\mathcal{O}$-operator} associated to $(M, \rho)$.
\end{defi}

%The relationships between skew-symmetric solutions of the CCYBE and $\mathcal{O}$-operators are as follows.
\begin{pro} \cite[Theorem 3.1]{HB}\mlabel{conformal operator1}
Let $R$ be a finite Lie conformal algebra which is free as a ${\bf k}[\partial]$-module and $r\in R\otimes R$ be skew-symmetric. Then $r$ is a solution of the CCYBE in $R$ if and only
$T_0^r=T_\lambda^r|_{\lambda=0}$ is an $\mathcal{O}$-operator associated to $(R^{\ast c}, \ad^\ast)$, where $T^r\in \text{Chom}(R^{\ast c}, R)$ is defined by
\vspc
\begin{eqnarray}
T_\lambda^r(u)=\sum_iu_{-\lambda-\partial}(x_i)y_i,\;\;\; \text{ $u\in R^{\ast c}$, where $r=\sum_i x_i\otimes y_i$.}
\end{eqnarray}
\end{pro}

\begin{pro}\cite[Theorem 3.7]{HB}\mlabel{conformal operator2}
Let $R$ be a finite Lie conformal algebra and $(M, \rho)$ be a representation of $R$. Suppose that
$R$ and $M$ are free as ${\bf k}[\partial]$-modules and $T: M\rightarrow R$ is a ${\bf k}[\partial]$-module homomorphism. Let $\{e_i\}_{i=1}^n$ be a ${\bf k}[\partial]$-basis of $R$, $\{v_i\}_{i=1}^m$ be a ${\bf k}[\partial]$-basis of $M$ and $\{v_i^\ast\}_{i=1}^m$ be the dual ${\bf k}[\partial]$-basis in $M^{\ast c}$. Then $T$ is an $\mathcal{O}$-operator associated to $(M, \rho)$ if and only if $r=r_T-\tau r_T$ is a skew-symmetric solution of the CCYBE in the semi-direct product $R\ltimes_{\rho^\ast} M^{\ast c}$, where
\vspc
\begin{eqnarray*}
r_T=\sum_{i=1}^m\sum_{j=1}^na_{ij}(0, \partial\otimes \id) e_j \otimes v_i^\ast,\;\;\; \text{if $T (v_i)=\sum_{j=1}^na_{ij}(0,\partial)e_j$.}
\vspb
\end{eqnarray*}
\end{pro}

Next, we give a correspondence between $\mathcal{O}$-operators on \gdas and $\mathcal{O}$-operators on the corresponding Lie conformal algebras which follows from a direct check.
\begin{pro}\mlabel{Conf-o-operator}
Let $(A, \circ, [\cdot,\cdot])$ be a \gda and $R={\bf k}[\partial]A$ be the Lie conformal algebra corresponding to $(A, \circ, [\cdot,\cdot])$. Suppose that $(V, l_A, r_A, \rho_A)$ is a representation of $(A, \circ, [\cdot,\cdot])$.
Then $T: V\rightarrow A$ is an $\mathcal{O}$-operator on $(A, \circ, [\cdot,\cdot])$ associated to $(V, l_A, r_A, \rho_A)$ if and only if the ${\bf k}[\partial]$-module homomorphism $\widetilde{T}: {\bf k}[\partial]V\rightarrow {\bf k}[\partial]A$ defined by
\vspb
\begin{eqnarray}
\widetilde{T}(v)=T(v),\;\;v\in V,
\end{eqnarray}
is an $\mathcal{O}$-operator on $R$ associated to $({\bf k}[\partial]V, \rho_{l_A, r_A, \rho_A})$.
\end{pro}
Similar to Corollary~\mref{cor:r}, there are the following two commutative diagrams.
\vspb
\begin{displaymath}
\xymatrix{
\txt{\tiny $r$\\ \tiny a skew-symmetric solution of \\ \tiny the GDYBE in $(A, \circ, [\cdot,\cdot])$ } \ar[rr]^-{\rm Prop. \mref{oper-form1}}   \ar@{<->}[d]_-{\rm Prop. \mref{conformal Yang-Baxter equation}}&& \txt{\tiny $T^r$\\ \tiny an $\mathcal{O}$-operator on \\ \tiny $(A, \circ, [\cdot,\cdot])$ associated to\\ \tiny $(A^\ast, L_{\star}^\ast, -R_{\circ}^\ast, \ad^\ast)$} \ar@{<->}[d]_-{\rm Prop. \mref{Conf-o-operator}}\\
\txt{\tiny $r$\\ \tiny a skew-symmetric solution of \\ \tiny the CCYBE in ${\bf k}[\partial]A$ } \ar[rr]^-{\rm
Prop. \mref{conformal operator1}}
            && \txt{\tiny $\widetilde{T^r}$\\ \tiny an $\mathcal{O}$-operator on ${\bf k}[\partial]A$ \\\tiny associated to $({\bf k}[\partial]A^\ast, \rho_{L_{\star}^\ast, -R_{\circ}^\ast, \ad^\ast})$}}
\vspb
\end{displaymath}
\begin{displaymath}
\xymatrix{
\txt{\tiny $T$\\ \tiny an $\mathcal{O}$-operator on $(A,\circ, [\cdot,\cdot])$ \\\tiny associated to
$(V,l_A,r_A, \rho_A)$ } \ar[rr]^-{\rm Prop. \mref{oper-form2}}   \ar@{<->}[d]_-{\rm  Prop. \mref{Conf-o-operator}}&& \txt{\tiny $r=r_T-\tau r_T$ \\ \tiny a skew-symmetric solution\\
\tiny of the GDYBE in
$(A\ltimes_{l_A^\ast+r_A^\ast,-r_A^\ast, \rho_A^\ast} V^\ast, \bullet, [\cdot,\cdot])$} \ar@{<->}[d]_-{\rm Prop. \mref{conformal Yang-Baxter equation}}\\
\txt{\tiny $\widetilde{T}$\\ \tiny an $\mathcal{O}$-operator on $R={\bf k}[\partial]A$ \\ \tiny associated to $({\bf k}[\partial]V^\ast, \rho_{l_A^\ast+r_A^\ast,-r_A^\ast, \rho_A^\ast})$} \ar[rr]^-{\rm
Prop. \mref{conformal operator2}}
            && \txt{\tiny $r=r_T-\tau r_T$\\ \tiny a skew-symmetric solution of\\ \tiny the CCYBE in $R \ltimes_{\rho_{l_A^\ast+r_A^\ast,-r_A^\ast, \rho_A^\ast}} {\bf k}[\partial]V^\ast$ }}
\vspb
\end{displaymath}

\begin{defi}\mcite{HL2}
A {\bf left-symmetric conformal algebra} $B$ is a ${\bf k}[\partial]$-module endowed with a ${\bf k}$-linear map $\cdot_\lambda \cdot: B\otimes B\rightarrow B[\lambda]$ such that
\vspb
$$(\partial a)_\lambda b=-\lambda a_\lambda b,\;\; a_\lambda (\partial b)=(\partial+\lambda) a_\lambda b,
 (a_\lambda b)_{\lambda+\mu}c -a_\lambda (b_\mu c)=(b_\mu a)_{\lambda+\mu}c-b_\mu (a_\lambda c),\;\;\; a, b\in B.
$$
\end{defi}

A correspondence between pre-\gdas and a class of left-symmetric conformal algebras is given as follows.
\begin{pro}\cite[Theorem 2.21]{XH}\mlabel{corr-left-symm-1}
Let $B={\bf k}[\partial]A$ be a free ${\bf k}[\partial]$-module
over a vector space $A$. Then $B$ is a left-symmetric
conformal algebra with the following $\lambda$-products \vspb
\begin{equation}\mlabel{115} a_\lambda
b=\partial (b\lhd a)+\lambda (a\rhd b+b\lhd a)+a\diamond b,\;\;\; ~~a, b\in
A,\end{equation}
where $\lhd$,  $\rhd$ and $\diamond$ are three binary operations on $A$ if and only if $(A, \lhd, \rhd, \diamond)$ is a pre-\gda. We call that $B$ is {\bf the left-symmetric conformal algebra corresponding to $(A, \lhd, \rhd, \diamond)$}.
\vspb
\end{pro}

\begin{pro}\mlabel{O-oper-left-symm}
\begin{enumerate}
\item  \cite[Proposition 2.5]{HL2} Let $B$ be a left-symmetric conformal algebra. Then the $\lambda$-brackets
\vspb
\begin{eqnarray}
[a_\lambda b]=a_\lambda b-b_{-\lambda-\partial}a,\;\;\; a, b \in B,
\end{eqnarray}
define a Lie conformal algebra structure on $B$, denoted by $\mathfrak{g}(B)$, which is called the {\bf sub-adjacent Lie conformal algebra of $B$}. Therefore, the identity map $\id$ on $B$ is an $\mathcal{O}$-operator on $\mathfrak{g}(B)$ associated with the representation $(B, L_B)$, where
\vspb
%\begin{eqnarray*}
%{L_B(a)}_\lambda b=a_\lambda b, \;\; a, b\in B.
%\end{eqnarray*}
%\cm{I label the above equation since we will use it later}
\begin{eqnarray}\label{eq:sss}
{L_B(a)}_\lambda b=a_\lambda b, \;\; a, b\in B.
\end{eqnarray}

\item \cite[Corollary 4.6]{HB}
Let $T$ be an $\mathcal{O}$-operator  on a Lie conformal algebra $R$ associated to a representation $(M, \rho)$. Then there is a left-symmetric conformal algebra structure on $M$ defined by
\vspc
\begin{eqnarray}
u_\lambda v=\rho(T(u))_\lambda v,\;\;\; u, v\in M.
\vspb
\end{eqnarray}
\end{enumerate}
\end{pro}

\begin{pro}\mlabel{pro-Lie-conf-1}
Let $(A, \lhd, \rhd, \diamond)$ be a pre-\gda and $(A, \circ, [\cdot,\cdot])$ be the associated \gda. Let $B$ be the left-symmetric conformal algebra corresponding to $(A, \lhd, \rhd, \diamond)$. Then the sub-adjacent Lie conformal algebra of $B$ is just the Lie conformal algebra corresponding to $(A, \circ, [\cdot,\cdot])$.
\vspb
\end{pro}

\begin{proof}
Let $a$, $b\in A$. Then we have
\vspb
\begin{eqnarray*}
[a_\lambda b]&=&a_\lambda b-b_{-\lambda-\partial}a\\
&=& \partial (b\lhd a)+\lambda (a\rhd b+b\lhd a)+a\diamond b-(\partial (a\lhd b)+(-\lambda-\partial) (b\rhd a+a\lhd b)+b\diamond a)\\
&=& \partial(b\lhd a+b\rhd a)+\lambda (a\rhd b+a\lhd b+b\rhd a+b\lhd a)+(a\diamond b-b\diamond a)\\
&=& \partial(b\circ a)+\lambda(a\star b)+[a,b].
\vspb
\end{eqnarray*}
Hence we obtain the conclusion.
\end{proof}

\begin{pro}
Let $R={\bf k}[\partial]A$ be the Lie conformal algebra corresponding to a \gda $(A,\circ, [\cdot,\cdot])$. Suppose that $T$ is an $\mathcal
O$-operator on $(A,\circ, [\cdot,\cdot])$ associated to a representation
$(V,l_A,r_A, \rho_A)$ and $(V,\rhd,\lhd, \diamond)$ be the pre-\gda defined by
Eq.~\meqref{eq:ndend}. Then the left-symmetric conformal structure on
${\bf k}[\partial]V$ induced from the $\mathcal O$-operator $\widetilde{T}$ on the Lie conformal algebra $R$ associated to the
representation $({\bf k}[\partial]V, \rho_{l_A, r_A, \rho_A})$ is exactly the one
corresponding to the pre-\gda $(V,\rhd,\lhd, \diamond)$.
\end{pro}

\begin{proof}
Let $u$, $v\in V$. Then we have
\vspb
\begin{eqnarray*}
u_\lambda v&=&\rho_{l_A,r_A,\rho_A}(\widetilde{T}(u))_\lambda v\\
&=&\partial (r_A(T(u))v)+\lambda(l_A(T(u))v+r_A(T(u))v)+\rho_A(T(u))v\\
&=& \partial (v\lhd u)+\lambda (u\rhd v+v\lhd u)+u\diamond v.
\vspb
\end{eqnarray*}
Therefore the conclusion holds.
\end{proof}

\begin{cor}
Let $(A, \lhd, \rhd, \diamond)$ be a pre-\gda and $(A, \circ,
[\cdot,\cdot])$ be the associated \gda. Let $B={\bf k}[\partial]A$
be the left-symmetric conformal algebra corresponding to $(A,
\lhd, \rhd, \diamond)$. Thus $\mathfrak{g}(B)$ is the Lie
conformal algebra corresponding to $(A, \circ, [\cdot,\cdot])$.
Then the $\mathcal O$-operator $\widetilde{\id_A}$ on the Lie
conformal algebra $R$ associated to $(R, \rho_{L_\rhd,R_\lhd,
L_\diamond})$ is exactly $\id_{\mathfrak{g}(B)}$ as the $\mathcal
O$-operator on $\mathfrak{g}(B)$ associated to the representation
$(B, L_B)$, where $L_B$ is defined by
Eq.~(\ref{eq:sss}).
\end{cor}
\begin{proof}
Obviously, $\widetilde{\id_A}=\id_{\mathfrak{g}(B)}$. Then it follows directly from Proposition \mref{pro-Lie-conf-1}.
\end{proof}

Furthermore, considering the sufficient and necessary conditions in Propositions~\mref{Conf-o-operator} and~\mref{corr-left-symm-1},
 we have the following commutative diagram.
\vspb
 \begin{displaymath}
\xymatrix{
\txt{\tiny $(A, \lhd, \rhd, \diamond)$\\ \tiny a pre-\gda } \ar[rr]^-{\rm Prop. \mref{GD-PGD} }  \ar@{<->}[d]_-{\rm Prop. \mref{corr-left-symm-1}}&& \txt{\tiny $\id_A$\\ \tiny an $\mathcal{O}$-operator on $(A, \circ, [\cdot,\cdot])$ \\ \tiny associated to $(A, L_\rhd, R_\lhd, L_{\diamond})$} \ar@{<->}[d]_-{\rm Prop. \mref{Conf-o-operator}}\\
\txt{\tiny $B={\bf k}[\partial]A$\\ \tiny a left-symmetric conformal algebra } \ar[rr]^-{\rm
Prop. \mref{O-oper-left-symm}}
            && \txt{ \tiny $\id_{\mathfrak{g}(B)}$\\ \tiny an $\mathcal{O}$-operator on $\mathfrak{g}(B)$ \\ \tiny associated to $(B, L_B)$ }}
\end{displaymath}

Combining the results above, we have

\begin{thm}\mlabel{Constr-Lie-Conf-pre-Nov}
Let $(A, \lhd, \rhd, \diamond)$ be a pre-\gda and $(A, \circ, [\cdot,\cdot])$ be the
associated \gda. Let $R={\bf k}[\partial]A$ be the Lie conformal algebra corresponding to $(A, \circ, [\cdot,\cdot])$. Define $r$ by Eq.~(\mref{eq:solu}), which is
a skew-symmetric
solution of the GDYBE in the \gda
$\hat A= A\ltimes_{L_\rhd^\ast+R_\lhd^\ast, -R_\lhd^\ast, L_\diamond^\ast}A^\ast$. Then there is a Lie conformal bialgebra structure $\delta$ on the Lie conformal algebra $\hat{R}={\bf k}[\partial]\hat A$ corresponding to the \gda $\hat A$ with $\delta$ defined by Eq.~(\mref{coboundary}).
Moreover, it is exactly the one corresponding to the \gdba $(\hat A, \circ, [\cdot,\cdot], \Delta, \delta_0)$ with $\delta$ defined by Eq.~(\mref{cobracket}), where $\Delta=\Delta_r$ and $\delta_0=\delta_r$ are defined by
Eqs. (\mref{coNov}) and (\mref{coLie}) respectively through $r$.
\end{thm}

We illustrate the utility of the previous constructions by
presenting an example of GD bialgebras and Lie conformal
bialgebras constructed from Zinbiel algebras with derivations.

\begin{ex}
Let $(A={\bf k }e_1\oplus {\bf k} e_2 \oplus {\bf k} e_3, \cdot)$ be the 3-dimensional Zinbiel algebra whose non-zero products are given by
\vspb
\begin{eqnarray*}
e_1\cdot e_1=e_2, \;\;e_1\cdot e_2=e_3,\;\; e_2\cdot e_1=e_3.
\end{eqnarray*}
Let $D: A\rightarrow A$ be the derivation given by
\vspb
\begin{eqnarray*}
D(e_1)=e_1,\;\;D(e_2)=2e_2,\;\;D(e_3)=3e_3.
\end{eqnarray*}
By Example \mref{constr-pre-GD}, there is a pre-\gda
$(A={\bf k }e_1\oplus {\bf k} e_2 \oplus {\bf k} e_3, \lhd, \rhd, \diamond)$   whose non-zero products are given by
\vspb
\begin{eqnarray*}
&&e_1\lhd e_1=(1+\xi)e_2,\;\;e_1 \lhd e_2=(2+\xi)e_3,\;\;e_2\lhd e_1=(1+\xi)e_3,\;\;e_1\rhd e_1=(1+\xi)e_2,\\
&&e_1\rhd e_2=(2+\xi)e_3,\;\; e_2\rhd e_1=(1+\xi)e_3,\;\;e_1\diamond e_2=ke_3,\;\;e_2\diamond e_1=-ke_3,
\end{eqnarray*}
for some $k$ and $\xi\in {\bf k}$. Then the associated \gda is $(A, \circ, [\cdot,\cdot])$ with the non-zero products  as follows.
\vspb
\begin{eqnarray*}
&&e_1\circ e_1=(2+2\xi)e_2,\;\;e_1 \circ e_2=(4+2\xi)e_3,\;\;e_2\circ e_1=(2+2\xi)e_3,\;\;[e_1, e_2]=2ke_3.
\end{eqnarray*}
 Let $\{e_1^\ast, e_2^\ast, e_3^\ast\}$ be the basis of $A^\ast$ dual to $\{e_1, e_2, e_3\}$. Then the non-zero products of the \gda $(\hat{A}=A\ltimes_{L_{\rhd}^\ast+R_{\lhd}^\ast,
-R_{\lhd}^\ast, L_\diamond^\ast}A^\ast, \bullet, [\cdot,\cdot])$ are given by
\begin{eqnarray*}
&&e_1\bullet e_1=(2+2\xi)e_2,\;\;e_1\bullet e_2=(4+2\xi)e_3,\;\;e_2\bullet e_1=(2+2\xi)e_3,\;\;e_1\bullet e_2^\ast=-(2+2\xi)e_1^\ast,\\
&&e_2^\ast\bullet e_1=(1+\xi)e_1^\ast,\;\;e_1\bullet e_3^\ast=-(3+2\xi)e_2^\ast,\;\;e_3^\ast\bullet e_1=(1+\xi)e_2^\ast,\;\;e_2\bullet e_3^\ast=-(3+2\xi)e_1^\ast,\\
&& e_3^\ast\bullet e_2=(2+\xi)e_1^\ast, \;\;[e_1, e_2]=2ke_3,\;\;[e_1,e_3^\ast]=-ke_2^\ast,\;\;[e_2,e_3^\ast]=ke_1^\ast.
\end{eqnarray*}
Set $r=\sum_{i=1}^3(e_i\otimes e_i^\ast-e_i^\ast\otimes e_i)$. By Theorem \mref{NYB-ND},  $r$ is a skew-symmetric solution of the GDYBE in $(\hat{A}=A\ltimes_{L_{\rhd}^\ast+R_{\lhd}^\ast,
-R_{\lhd}^\ast, L_\diamond^\ast}A^\ast, \bullet, [\cdot,\cdot])$. Then by Theorem \mref{Constr-Lie-Conf-pre-Nov}, there is a \gdba $(\hat{A}=A\ltimes_{L_{\rhd}^\ast+R_{\lhd}^\ast,
-R_{\lhd}^\ast, L_\diamond^\ast}A^\ast, \bullet, [\cdot,\cdot], \Delta, \delta_0)$ where $\Delta$ and $\delta_0$ are given by
\begin{eqnarray*}
&&\Delta(e_1)=-(1+\xi)e_2\otimes e_1^\ast-(2+\xi)e_3\otimes e_2^\ast+(2+2\xi)e_1^\ast\otimes e_2+(3+2\xi)e_2^\ast\otimes e_3,\\
&&\Delta(e_2)=-(1+\xi)e_3\otimes e_1^\ast+(3+2\xi)e_1^\ast\otimes e_3,\;\;\Delta(e_3)=\Delta(e_1^\ast)=0,\\
&&\Delta(e_2^\ast)=-(2+2\xi)e_1^\ast\otimes e_1^\ast,\;\;\Delta(e_3^\ast)=-(4+2\xi)e_1^\ast\otimes e_2^\ast-(2+2\xi)e_2^\ast\otimes e_1^\ast,\\
&&\delta_0(e_1)=2k(e_3\otimes e_2^\ast-e_2^\ast\otimes e_3)+k(e_2^\ast\otimes e_3-e_3\otimes e_2^\ast),\\
&&\delta_0(e_2)=k(e_3\otimes e_1^\ast-e_1^\ast\otimes e_3)+2k(e_1^\ast\otimes e_3-e_3\otimes e_1^\ast),\\
&&\delta_0(e_3)=\delta_0(e_1^\ast)=\delta_0(e_2^\ast)=0,\;\;\delta_0(e_3^\ast)=2k(e_2^\ast\otimes e_1^\ast-e_1^\ast\otimes e_2^\ast).
\end{eqnarray*}
By Theorem \mref{Constr-Lie-Conf-pre-Nov}, there is a Lie conformal bialgebra structure on ${\bf k}[\partial]\hat{A}$ with the non-zero $\lambda$-brackets and $\delta$ given by
\begin{eqnarray*}
[{e_1}_\lambda e_1]&=&(2+2\xi)(\partial+2\lambda)e_2,\;\;[{e_1}_\lambda e_2]=(2+2\xi)\partial e_3+(6+4\xi)\lambda e_3+2ke_3,\\
{} [{e_1}_\lambda e_2^\ast]&=& (1+\xi)(\partial-\lambda)e_1^\ast,\;\;[{e_1}_\lambda e_3^\ast]=(1+\xi)\partial e_2^\ast-(2+\xi)\lambda e_2^\ast-ke_2^\ast,\\
{}[{e_2}_\lambda e_3^\ast]&=&(2+\xi)\partial e_1^\ast-(1+\xi)\lambda e_1^\ast+ke_1^\ast,\\
\delta(e_3)&=&\delta(e_1^\ast)=0,\;\; \delta(e_2^\ast)=-(2+2\xi)(\partial e_1^\ast\otimes e_1^\ast-e_1^\ast\otimes \partial e_1^\ast),\\
\delta(e_1)&=&2k(e_3\otimes e_2^\ast-e_2^\ast\otimes e_3)+k(e_2^\ast\otimes e_3-e_3\otimes e_2^\ast)-(1+\xi)(\partial e_2\otimes e_1^\ast-e_1^\ast\otimes \partial e_2)\\
&&-(2+\xi)(\partial e_3\otimes e_2^\ast-e_2^\ast\otimes \partial e_3)+(2+2\xi)(\partial e_1^\ast\otimes e_2-e_2\otimes \partial e_1^\ast)\\
&&+(3+2\xi)(\partial e_2^\ast\otimes e_3-e_3\otimes \partial e_2^\ast),\\
\delta(e_2)&=&k(e_3\otimes e_1^\ast-e_1^\ast\otimes e_3)+2k(e_1^\ast\otimes e_3-e_3\otimes e_1^\ast)-(1+\xi)(\partial e_3\otimes e_1^\ast-e_1^\ast\otimes \partial e_3)\\
&& +(3+2\xi)(\partial e_1^\ast\otimes e_3-e_3\otimes \partial e_1^\ast),\\
\delta(e_3^\ast)&=&2k(e_2^\ast\otimes e_1^\ast-e_1^\ast\otimes e_2^\ast)-(4+2\xi)(\partial e_1^\ast\otimes e_2^\ast-e_2^\ast\otimes\partial e_1^\ast)-(2+2\xi)(\partial e_2^\ast\otimes e_1^\ast-e_1^\ast\otimes \partial e_2^\ast).
\vspb
\end{eqnarray*}
\end{ex}

\section{Characterizations of \gdbas and relations with Lie conformal bialgebras}
\mlabel{s:char}
In this section, we present two characterizations of \gdbas in terms of matched pairs and Manin triples of \gdas respectively.  Furthermore, the relations of the notions of matched pairs as well as Manin triples of \gdas with the corresponding notions of the Lie conformal algebras are established.

\subsection{Equivalent characterizations of \gdbas}

Next, we recall the definitions of matched pairs of Novikov algebras and Lie algebras.

\begin{defi}\mlabel{Matched pair} \mcite{Hong}
    Let $(A,\circ_A)$ and $(B,\circ_B)$ be Novikov algebras.  Let $l_A$, $r_A: A \rightarrow {\rm
        End}_{\bf k}(B)$ and $l_B$, $r_B: B \rightarrow {\rm
        End}_{\bf k}(A)$ be linear maps.
    If there is a Novikov algebra structure on the direct sum $A\oplus B$ of the underlying vector spaces of $A$ and $B$ given by
    {\small \begin{eqnarray}
            \mlabel{eq-Nov-match}&&(a+x)\bullet (b+y):=\big(a\circ_A b+l_B(x)b+r_B(y)a\big)+\big(x\circ_B y+l_A(a)y+r_A(b)x\big),\;a, b\in A, x, y\in B,
    \end{eqnarray}}
    then we call the resulting sextuple $(A,
    B, l_A, r_A, l_B, r_B)$ a {\bf matched pair of Novikov algebras}.
\end{defi}

\begin{defi}\mcite{M}
    Let $(A, [\cdot,\cdot]_A)$ and $(B, [\cdot,\cdot]_B)$ be Lie algebras, $\rho_A: A \rightarrow {\rm
        End}_{\bf k}(B)$ and $\rho_B: B \rightarrow {\rm
        End}_{\bf k}(A)$ be linear maps.
If there is a Lie algebra structure on the vector space $A\oplus B$ given by
    {\small \begin{eqnarray}
            \mlabel{eq-Lie-match}&&[a+x, b+y]:=([a,b]_A+\rho_{B}(x)b-\rho_{B}(y)a)+([x,y]_B+\rho_{A}(a)y-\rho_{A}(b)x),\;\; a, b\in A, x, y\in B,
    \end{eqnarray}}
    then we call the resulting quadruple $(A,
    B, \rho_{A}, \rho_{B})$ a {\bf matched pair of Lie algebras}.
\end{defi}

\begin{defi}
    Let $(A, \circ_A, [\cdot,\cdot]_A)$ and $(B, \circ_B, [\cdot,\cdot]_B)$ be \gdas. Let $l_A$, $r_A$, $\rho_A: A \rightarrow {\rm
        End}_{\bf k}(B)$ and $l_B$, $r_B$, $\rho_B: B \rightarrow {\rm
        End}_{\bf k}(A)$ be linear maps.
    If there is a \gda $(A\oplus B, \bullet, [\cdot,\cdot])$ where $\bullet$ and $[\cdot,\cdot]$ on $A\oplus B$ are given by Eqs. (\mref{eq-Nov-match}) and  (\mref{eq-Lie-match}) respectively, then we call the resulting octuple $(A, B, l_A, r_A, \rho_{A}, l_B, r_B, \rho_{B})$ a {\bf matched pair of \gdas}.
\end{defi}

\begin{pro}\cite[Example 3.6, Theorem 3.7]{WH}\mlabel{matched-pair-GD}
The octuple $(A, B, l_A, r_A, \rho_{A}, l_B, r_B, \rho_{B})$ is a  matched pair of \gdas if and only if $(A, B, l_A, r_A, l_B, r_B)$ is a matched pair of Novikov algebras, $(A, B, \rho_A, \rho_B)$ is a matched pair of Lie algebras, $(A, l_B, r_B, \rho_B)$ be a representation of $(B, \circ_B, [\cdot,\cdot]_B)$,  $(B, l_A, r_A, \rho_A)$ be a representation of $(A, \circ_A, [\cdot,\cdot]_A)$ and they also satisfy the compatibility conditions as follows.
    \begin{eqnarray}
        &&\mlabel{match1}[a, r_B(x)b]_A-\rho_B(l_A(b)x)a-\rho_B(x)(b\circ_A a)+r_B(x)[b,a]_A+(\rho_B(x)b)\circ_A a\\
        &&\;\;\;\;-l_B(\rho_A(b)x)a+b\circ_A (\rho_B(x)a)-r_B(\rho_A(a)x)b=0,\nonumber\\
        &&\mlabel{match2}[a, l_B(x)b]_A-\rho_B(r_A(b)x)a-[b, l_B(x)a]_A+\rho_B(r_A(a)x)b+(\rho_B(x)a)\circ_A b\\
        &&\;\;\;\;-l_B(\rho_A(a)x)b-(\rho_B(x)b)\circ_A a+l_B(\rho_A(b)x)a-l_B(x)[a,b]_A=0,\nonumber\\
        &&\mlabel{match3}[y,r_A(a)x]_B-\rho_A(a)(x\circ_B y)-\rho_A(l_B(x)a)y+(\rho_A(a)x)\circ_B y-l_A(\rho_B(x)a)y\\
        &&\;\;\;\;+r_A(a)[x,y]_B+x\circ_B (\rho_A(a)y)-r_A(\rho_B(y)a)x=0,\nonumber\\
        &&\mlabel{match4}[x, l_A(a)y]_B-\rho_A(r_B(y)a)x-[y, l_A(a)x]_B+\rho_A(r_B(x)a)y+(\rho_A(a)x)\circ_B y\\
        &&\;\;\;\;-l_A(\rho_B(x)a)y-(\rho_A(a)y)\circ_B x+l_A(\rho_B(y)a)x-l_A(a)[x,y]_B=0,\nonumber \;\;a, b\in A, x, y\in B.
    \end{eqnarray}
    Moreover, any \gda that can be decomposed into a linear direct sum of two GD subalgebras is obtained from a matched pair of \gdas.
\end{pro}

\begin{defi}\mlabel{Novbilinear}\mcite{HBG}
Let $(A,\circ)$ be a Novikov algebra. A bilinear form
$(\cdot,\cdot)$  on $A$ is called {\bf invariant} if it satisfies
\vspc
\begin{eqnarray}\mlabel{bilinear1}
        (a\circ b,c)=-(b, a\star c),\;\;\;
        a,b,c\in A.
        \vspa
\end{eqnarray}
A {\bf quadratic Novikov algebra}, denoted by $(A,\circ, (\cdot,\cdot))$, is a Novikov algebra $(A,\circ)$ together with a nondegenerate symmetric invariant bilinear form    $(\cdot,\cdot)$.
\end{defi}

Recall that a bilinear form $(\cdot,\cdot)_L $ on a Lie
algebra $L$ is called {\bf invariant} if
\begin{eqnarray}
    ( [a, b], c)_L=( a, [b, c])_L, \;\;
    a, b, c\in L.
\end{eqnarray}

\begin{defi}
    Let $(A, \circ, [\cdot,\cdot])$ be a \gda. If there is a bilinear form $(\cdot,\cdot)$ such that $(A, \circ, (\cdot,\cdot))$ is a quadratic Novikov algebra and $(\cdot,\cdot)$ is invariant on the Lie algebra $(A, [\cdot,\cdot])$, then $(A, \circ, [\cdot,\cdot], (\cdot,\cdot))$ is called a {\bf quadratic \gda}.
\end{defi}

\begin{defi}
A {\bf(standard) Manin triple of \gdas} is a tripe of \gdas $(A=A_1\oplus A_1^\ast, (A_1, \circ_{A_1}, [\cdot,\cdot]_{A_1}), (A_1^\ast, \circ_{A_1^\ast}, [\cdot,\cdot]_{A_1^\ast}))$ for which
\begin{enumerate}
\item $A$ is the vector space direct sum of $A_1$ and $A_1^\ast$;
\item $(A_1, \circ_{A_1}, [\cdot,\cdot]_{A_1})$ and $(A_1^\ast, \circ_{A_1^\ast}, [\cdot,\cdot]_{A_1^\ast})$ are GD subalgebras of $A$;
\item the bilinear form on $A_1\oplus A_1^*$ given by
\begin{eqnarray}\mlabel{bilinear}
(a+f,b+g)=\langle f, b\rangle+\langle g, a\rangle,\;\;\;\;a, b\in A_1, f, g\in A_1^\ast,
\end{eqnarray}
        is invariant on both $(A, \circ)$ and $(A, [\cdot,\cdot])$, that is, $(A, \circ, $ $[\cdot,\cdot], (\cdot,\cdot))$ is a quadratic \gda.
    \end{enumerate}
\end{defi}

Next, we present an equivalence between Manin triples of \gdas and some matched pairs of \gdas.
\begin{thm}\mlabel{equi-1}
    Let $(A, \circ_A, [\cdot,\cdot]_A)$ be a \gda. Suppose that there is a \gda structure $(A^\ast, \circ_{A^\ast}, [\cdot,\cdot]_{A^\ast})$ on $A^\ast$. Then there is a Manin triple $(A\oplus A^\ast, (A, \circ_A, [\cdot,\cdot]_A), $ $(A^\ast, \circ_{A^\ast}, [\cdot,\cdot]_{A^\ast}))$ if and only if $(A, A^\ast, L_{\star_A}^\ast, $ $-R_{\circ_A}^\ast, \ad_A^\ast,  L_{\star_{A^\ast}}^\ast, -R_{\circ_{A^\ast}}^\ast, \ad_{A^\ast}^\ast)$ is a matched pair of \gdas.
\end{thm}
\begin{proof}
Suppose that $(A, A^\ast, L_{\star_A}^\ast, $ $-R_{\circ_A}^\ast, \ad_A^\ast,  L_{\star_{A^\ast}}^\ast, -R_{\circ_{A^\ast}}^\ast, \ad_{A^\ast}^\ast)$ is a matched pair of \gdas. Then by the definition of matched pair, there is a \gda $(A\oplus A^\ast, \bullet, $ $[\cdot,\cdot])$ such that $(A, \circ_{A},$ $ [\cdot,\cdot]_{A})$ and $(A^\ast, \circ_{A^\ast}, [\cdot,\cdot]_{A^\ast})$ are its GD subalgebras. Moreover,
$(A, A^\ast, L_{\star_A}^\ast, -R_{\circ_A}^\ast, L_{\star_{A^\ast}}^\ast,$ $ -R_{\circ_{A^\ast}}^\ast)$ is a matched pair of Novikov algebras and $(A, A^\ast, \ad_A^\ast, \ad_{A^\ast}^\ast)$ is a matched pair of Lie algebras. Then by \cite[Theorem 3.10]{HBG}, the bilinear form $(\cdot,\cdot)$ defined by Eq. (\mref{bilinear}) is invariant on $(A\oplus A^\ast, \bullet)$. Similarly, since $(A, A^\ast, \ad_A^\ast, \ad_{A^\ast}^\ast)$ is a matched pair of Lie algebras, the bilinear form $(\cdot,\cdot)$ defined by Eq. (\mref{bilinear}) is also invariant on $(A\oplus A^\ast, [\cdot,\cdot])$. Therefore, $(A\oplus A^\ast, (A, \circ_A, [\cdot,\cdot]_A), $ $(A^\ast, \circ_{A^\ast}, [\cdot,\cdot]_{A^\ast}))$ is a Manin triple of \gdas.

Conversely, suppose that $(A\oplus A^\ast, (A, \circ_A, [\cdot,\cdot]_A), $ $(A^\ast, \circ_{A^\ast}, [\cdot,\cdot]_{A^\ast}))$ is a Manin triple of \gdas. Then by Proposition \mref{matched-pair-GD}, the \gda $(A\oplus A^\ast, \circ, [\cdot,\cdot])$ is obtained from some matched pair $(A, A^\ast, l_A, r_A, \rho_A, l_{A^\ast}, r_{A^\ast}, \rho_{A^\ast})$. Since $(\cdot,\cdot)$ defined by Eq. (\mref{bilinear}) is invariant on $(A\oplus A^\ast, \circ)$, we have $l_A=L_{\star_A}^\ast$, $r_A=-R_{\circ_A}^\ast$, $l_{A^\ast}=L_{\star_{A^\ast}}^\ast$ and $r_{A^\ast}=-R_{\circ_{A^\ast}}^\ast$ by \cite[Theorem 3.10]{HBG}. Similarly, since $(\cdot,\cdot)$ defined by Eq. (\mref{bilinear}) is invariant on $(A\oplus A^\ast, [\cdot,\cdot])$, we obtain $\rho_A=\ad_A^\ast$ and $\rho_{A^\ast}=\ad_{A^\ast}^\ast$. Then the proof is completed.
\end{proof}

Finally, we give an equivalence between certain matched pairs of \gdas and \gdbas.
\begin{thm}\mlabel{equi-2}
    Let $(A, \circ_A, [\cdot,\cdot]_A)$ be a \gda. Suppose that there is a \gda structure $(A^\ast,\circ_{A^\ast}, [\cdot,\cdot]_{A^\ast})$ and let $(A,\Delta,\delta)$ be the corresponding dual GD coalgebra.
Then $(A, A^\ast, L_{\star_A}^\ast, $ $-R_{\circ_A}^\ast, \ad_A^\ast,  L_{\star_{A^\ast}}^\ast, -R_{\circ_{A^\ast}}^\ast, \ad_{A^\ast}^\ast)$ is a matched pair of \gdas if and only if $(A, \circ_A, [\cdot,\cdot]_A, \Delta, \delta)$ is a \gdba.
\end{thm}
\begin{proof}
    Note that $(A, A^\ast, \ad_A^\ast, \ad_{A^\ast}^\ast)$ is a matched pair of Lie algebras if and only if $(A, [\cdot,\cdot]_A, \delta)$ is a Lie bialgebra.
    Moreover, $(A, A^\ast, L_{\star_A}^\ast, -R_{\circ_A}^\ast, L_{\star_{A^\ast}}^\ast, -R_{\circ_{A^\ast}}^\ast)$ is a matched pair of Novikov algebras if and only if $(A, \circ_A, \Delta)$ is a Novikov bialgebra by \cite[Theorem 3.11]{HBG}. Therefore we only need to prove that Eqs. (\mref{match1})-(\mref{match4}) hold if and only if Eq. (\mref{Lb4}) holds when $l_A=L_{\star_A}^\ast$, $r_A=-R_{\circ_A}^\ast$, $l_B=L_{\star_{A^\ast}}^\ast$, $r_B=-R_{\circ_{A^\ast}}^\ast$, $\rho_A=\ad_A^\ast$ and $\rho_B=\ad_{A^\ast}^\ast$.

    Let $\{e_1,\cdots, e_n\}$ be a basis of $A$ and $\{e_1^\ast, \cdots, e_n^\ast\}$ be its dual basis. Let
\vspb
    \begin{eqnarray*}
        e_i\circ_A e_j=\sum_{k=1}^na_{ij}^ke_k,\;\;[e_i, e_j]_A=\sum_{k=1}^nb_{ij}^ke_k,\;\;\Delta(e_k)=\sum_{i,j=1}^nc_{ij}^ke_i\otimes e_j,\;\;\delta(e_k)=\sum_{i,j=1}^nd_{ij}^ke_i\otimes e_j.
\vspb
    \end{eqnarray*}
Consequently, $e_i^\ast \circ_{A^\ast} e_j^\ast=\sum_{k=1}^n c_{ij}^ke_k^\ast$ and $[e_i^\ast, e_j^\ast]_{A^\ast}=\sum_{k=1}^n d_{ij}^ke_k^\ast$. Moreover, we obtain
\vspb
    \begin{eqnarray*}
        &&L_{\star_A}^\ast(e_i)e_j^\ast=-\sum_{k=1}^n(a_{ik}^j+a_{ki}^j)e_k^\ast,\;\; -R_{\circ_A}^\ast(e_i)e_j^\ast=\sum_{k=1}^na_{ki}^je_k^\ast,\;\;L_{\star_{A^\ast}}^\ast(e_i^\ast)e_j=-\sum_{k=1}^n(c_{ik}^j+c_{ki}^j)e_k,\\
        && -R_{\circ_{A^\ast}}^\ast(e_i^\ast)e_j=\sum_{k=1}^nc_{ki}^je_k,\;\; \ad_A^\ast(e_i)e_j^\ast=-\sum_{k=1}^nb_{ik}^je_k^\ast,\;\;  \ad_{A^\ast}^\ast(e_i^\ast)e_j=-\sum_{k=1}^nd_{ik}^je_k.
\vspb
    \end{eqnarray*}

    Let $a=e_i$, $x=e_k^\ast$ and $b= e_j$ in Eqs. (\mref{match1}) and (\mref{match2}). Comparing the coefficients of $e_t$ gives
\vspb
\begin{eqnarray}
&&\mlabel{eq:match1}\;\;\sum_{s=1}^n\Big(c_{sk}^jb_{is}^t-(a_{js}^k+a_{sj}^k)d_{st}^i+a_{ji}^sd_{kt}^s+b_{ji}^sc_{tk}^s-d_{ks}^ja_{si}^t-b_{js}^k(c_{st}^i+c_{ts}^i)-d_{ks}^ia_{js}^t+b_{is}^kc_{ts}^j\Big)=0,\\
\vspb
&&\mlabel{eq:match2}\;\;\sum_{s=1}^n\Big(-(c_{ks}^j+c_{sk}^j)b_{is}^t+a_{sj}^kd_{st}^i+(c_{ks}^i+c_{sk}^i)b_{js}^t-a_{si}^kd_{st}^j-d_{ks}^ia_{sj}^t-b_{is}^k(c_{st}^j+c_{ts}^j)\\
&&\qquad\qquad+d_{ks}^ja_{si}^t+b_{js}^k(c_{st}^i+c_{ts}^i)+b_{ij}^s(c_{kt}^s+c_{tk}^s)\Big)=0.\nonumber
\end{eqnarray}
Let $a=e_i$, $x=e_j^\ast$ and $y=e_k^\ast$ in Eqs. (\mref{match3}) and (\mref{match4}). Comparing the coefficients of $e_t^\ast$, we obtain
\vspb
\begin{eqnarray}
&&\mlabel{eq:match3}\sum_{s=1}^n\Big(a_{si}^jd_{ks}^t+c_{jk}^sb_{it}^s-(c_{js}^i+c_{sj}^i)b_{st}^k-b_{is}^jc_{sk}^t-d_{js}^i(a_{st}^k+a_{ts}^k)+d_{jk}^sa_{ti}^s-b_{is}^kc_{js}^t+d_{ks}^ia_{ts}^j\Big)=0,\\
\vspb
&&\mlabel{eq:match4}\sum_{s=1}^n\Big(-(a_{is}^k+a_{si}^k)d_{js}^t+c_{sk}^ib_{st}^j+(a_{is}^j+a_{si}^j)d_{ks}^t-c_{sj}^ib_{st}^k-b_{is}^jc_{sk}^t-d_{js}^i(a_{st}^k+a_{ts}^k)\\
&&\qquad\qquad\qquad \qquad  +b_{is}^kc_{sj}^t+d_{ks}^i(a_{st}^j+a_{ts}^j)+d_{jk}^s(a_{it}^s+a_{ti}^s)\Big)=0.\nonumber
\end{eqnarray}

Let $a=e_i$ and $b=e_j$ in Eq. (\mref{Lb4}). Comparing the coefficients of $e_s\otimes e_t$, we get
\vspb
\begin{eqnarray}
&&\mlabel{eq:match6}\sum_{k=1}^n\Big(a_{ji}^kd_{st}^k+b_{ij}^kc_{st}^k-d_{kt}^ja_{ki}^s-d_{kt}^ia_{jk}^s-d_{sk}^i(a_{jk}^t+a_{kj}^t)-c_{kt}^jb_{ik}^s-c_{sk}^jb_{ik}^t+b_{jk}^t(c_{sk}^i+c_{ks}^i)\Big)=0.
\end{eqnarray}
It is straightforward to show that Eq. (\mref{eq:match3}) holds if and only if Eq. (\mref{eq:match1}) holds. By Eq. (\mref{eq:match1}), we get
\vspb
    \begin{eqnarray*}
        \sum_{s=1}^n(a_{js}^k+a_{sj}^k)d_{st}^i=\sum(c_{sk}^jb_{is}^t+a_{ji}^sd_{kt}^s+b_{ji}^sc_{tk}^s-d_{ks}^ja_{si}^t-b_{js}^k(c_{st}^i+c_{ts}^i)-d_{ks}^ia_{js}^t+b_{is}^kc_{ts}^j).
\vspb
    \end{eqnarray*}
    Substituting it into the expression on the left hand side of Eq. (\mref{eq:match4}), we find that Eq. (\mref{eq:match4}) holds. Similarly, we also get that Eq. (\mref{eq:match2}) holds by Eq. (\mref{eq:match1}). Therefore, Eqs. (\mref{eq:match1})-(\mref{eq:match4}) hold if and only if  Eq. (\mref{eq:match1}) holds.
    Replacing $k$, $s$ and $t$ in Eq. (\mref{eq:match6}) by $s$, $t$ and $k$ respectively, we get that Eq. (\mref{eq:match1}) holds if and only if Eq. (\mref{eq:match6}) holds. Then the proof is completed.
\end{proof}

Theorems \mref{equi-1} and \mref{equi-2} give the following characterization of \gdbas.
\begin{cor}\mlabel{GD-equi}
Let $(A, \circ_A, [\cdot,\cdot]_A)$ be a \gda. Suppose that there is a \gda structure $(A^\ast, \circ_{A^\ast},[\cdot,\cdot]_{A^\ast})$ on $A^\ast$ obtained from $\Delta:A\rightarrow A\otimes A$ and $\delta: A\rightarrow A\otimes A$ respectively. Then the following conditions are equivalent.
\begin{enumerate}
\item $(A, \circ_A, [\cdot,\cdot]_A, \Delta, \delta)$ is a \gdba;
\item There is a Manin triple $(A\oplus A^\ast, (A, \circ_A, [\cdot,\cdot]_A), (A^\ast, \circ_{A^\ast}, [\cdot,\cdot]_{A^\ast}))$;
\item $(A, A^\ast, L_{\star_A}^\ast, $ $-R_{\circ_A}^\ast, \ad_A^\ast,  L_{\star_{A^\ast}}^\ast, -R_{\circ_{A^\ast}}^\ast, \ad_{A^\ast}^\ast)$ is a matched pair of \gdas.
\end{enumerate}
\end{cor}

\subsection{Correspondences between \gdbas and Lie conformal bialgebras in terms of matched pairs and Manin triples}

Now recall matched pairs of Lie conformal algebras.
\begin{defi} \mcite{HL}\mlabel{conf-matched-pair}
Let $R_1$ and $R_2$ be Lie conformal algebras, $\rho: R_1\rightarrow \text{gc}(R_2)$ and $\sigma: R_2\rightarrow \text{gc}(R_1)$ be ${\bf k}[\partial]$-module homomorphisms.
If there is a Lie conformal algebra structure on the direct sum $R_1\oplus R_2$ as ${\bf k}[\partial]$-modules given by
\begin{eqnarray}\mlabel{107}
&&[(a+x)_\lambda (b+y)]=([a_\lambda b]+\sigma(x)_\lambda b-\sigma(y)_{-\lambda-\partial}a)+([x_\lambda y]+\rho(a)_\lambda y-\rho(b)_{-\lambda-\partial} x),\;\;
\end{eqnarray}
for all $a, b\in R_1$, and $x, y\in R_2$, then $(R_1, R_2, \rho, \sigma)$ is called {\bf a matched pair of Lie conformal algebras}.
\end{defi}

\begin{rmk}
    Let $(R_1,R_2,\rho,\sigma)$ be a matched pair of Lie conformal algebras. If $R_2$ is abelian and $\sigma$ is trivial, then the Lie conformal algebra on the ${\bf k}[\partial]$-module $R_1\oplus R_2$ given by Eq. (\mref{107}) is called the {\bf semi-direct product of $R_1$ and its representation $(R_2, \rho)$}. We denote it by $R_1\ltimes_\rho R_2$.
\end{rmk}

Next, a correspondence between matched pairs of \gdas and certain matched pairs of the corresponding Lie conformal algebras is given as follows.

\begin{pro}\mlabel{corr-matched pair}
Let ${\bf k}[\partial]A$ and ${\bf k}[\partial]B$ be the Lie conformal algebras corresponding to \gdas $(A, \circ_A, [\cdot,\cdot]_A)$ and $(B, \circ_B, [\cdot,\cdot]_B)$ respectively. Let $l_A$, $r_A$, $\rho_A: A\rightarrow \text{End}_{{\bf k}}(B)$ and $l_B$, $r_B$, $\rho_B: B\rightarrow \text{End}_{{\bf k}}(A)$ be linear maps. Then $({\bf k}[\partial]A, {\bf k}[\partial]B, \rho_{l_A,r_A, \rho_A}, \rho_{l_B, r_B, \rho_B})$ is a matched pair of Lie conformal algebras if and only if $(A, B, l_A, r_A, \rho_A, l_B, r_B, \rho_B)$ is a matched pair of \gdas.
\end{pro}
\begin{proof} The proof follows from a similar argument as the one for Proposition \ref{corr-resp}.
\end{proof}

\begin{cor}\mlabel{equi-matched pair}
    Let $R={\bf k}[\partial]A$ and $R^{\ast c}={\bf k}[\partial]A^\ast$ be the Lie conformal algebras corresponding to \gdas $(A, \circ_A, [\cdot,\cdot]_A)$ and $(A^\ast, \circ_{A^\ast}, [\cdot,\cdot]_{A^\ast})$ respectively. Then $(R, R^{\ast c}, \ad_R^\ast, \ad_{R^{\ast c}}^\ast)$ is a matched pair of Lie conformal algebras if and only if $(A, A^\ast, L_{\star_A}^\ast, -R_{\circ_A}^\ast, \ad_A^\ast, L_{ \star_{A^\ast}}^\ast, -R_{\circ_{A^\ast}}^\ast, \ad_{A^\ast}^\ast)$ is a matched pair of \gdas.
\end{cor}
\begin{proof}
    Note that $\ad_R=\rho_{L_{\circ_A}, R_{\circ_A}, \ad_A}$. Therefore, by Proposition \mref{dual resp}, $\ad_R^\ast=\rho_{ L_{\star_A}^\ast, -R_{\circ_A}^\ast, \ad_A^\ast}$. Similarly, $\ad_{R^{\ast c}}^\ast=\rho_{ L_{\star_{A^\ast}}^\ast, -R_{\circ_{A^\ast}}^\ast, \ad_A^\ast}$.
    Then the conclusion follows from Proposition \mref{corr-matched pair}.
\end{proof}

Let $V$ be a ${\bf k}[\partial]$-module. Recall \mcite{L} that a {\bf conformal bilinear form} on $V$ is a ${\bf k}$-bilinear map $\langle \cdot, \cdot \rangle_\lambda: V\times V\rightarrow {\bf k}[\lambda]$ such that
\vspb
\begin{eqnarray}
    \langle \partial u, v \rangle_\lambda=-\lambda \langle u, v \rangle_\lambda=-\langle u, \partial v \rangle_\lambda,\;\;\; u, v\in V.
\end{eqnarray}
A conformal bilinear form is called {\bf symmetric} if $\langle u, v \rangle_\lambda=\langle v, u \rangle_{-\lambda}$ for all $u$, $v\in V$.
A conformal bilinear form on a Lie conformal algebra $R$ is called {\bf invariant} if
\vspb
\begin{eqnarray}
    \langle [a_\mu b], c \rangle_\lambda=\langle a, [b_{\lambda-\partial} c] \rangle_\mu,\;\;\;\;a, b, c\in R.
\end{eqnarray}

Given a conformal bilinear form on a ${\bf k}[\partial]$-module $V$. Then we have a ${\bf k}[\partial]$-module homomorphism $L: V\rightarrow V^{\ast c}$, $v\rightarrow L_v$ given by
\begin{eqnarray*}
    (L_v)_\lambda u=\langle v, u\rangle_\lambda, \;\;\;u\in V.
\end{eqnarray*}
If $L$ gives an isomorphism between $V$ and $V^{\ast c}$, then we say that this conformal bilinear form is {\bf nondegenerate}.

A correspondence between quadratic \gdas and Lie conformal algebras with a symmetric invariant nondegenerate conformal bilinear form is given as follows.
\begin{pro}\mlabel{corresp-bilinear}
    Let $R={\bf k}[\partial]A$ be the Lie conformal algebra corresponding to a \gda $(A, \circ, [\cdot,\cdot])$. Suppose that there is a bilinear form $(\cdot, \cdot): A\times A\rightarrow {\bf k}$.
    Define a conformal bilinear form $\langle \cdot, \cdot\rangle_\lambda$ on $R$ by
\vspc
    \begin{eqnarray}
        \langle a, b\rangle_\lambda=(a,b),\;\;\; a, b \in A.
    \end{eqnarray}
    Then $\langle \cdot, \cdot \rangle_\lambda$ is a symmetric invariant nondegenerate conformal bilinear form if and only if $(A, \circ, [\cdot,\cdot], (\cdot, \cdot))$ is a quadratic \gda.
\end{pro}
\begin{proof}
    Obviously, $\langle \cdot, \cdot \rangle_\lambda$ is symmetric and nondegenerate if and only if $(\cdot, \cdot)$ is symmetric and nondegenerate.
    Let $a$, $b$ and $c\in A$. Then we have
\vspb
\begin{eqnarray*}
        &&\langle [a_\mu b], c \rangle_\lambda = \langle \partial(b\circ a)+\mu (a\star b)+[a,b], c \rangle_\lambda=-\lambda (b\circ a, c)+\mu (a\star b, c)+([a,b],c),\\
        &&\langle a, [b_{\lambda-\partial} c] \rangle_\mu =\langle a, \partial (c\circ b)+(\lambda-\partial)(b\star c) +[b,c]\rangle_\mu
        =\lambda (a, b\star c)-\mu (a, b\circ c)+(a, [b,c]).
    \end{eqnarray*}
Therefore, $\langle [a_\mu b], c \rangle_\lambda=\langle a, [b_{\lambda-\partial} c] \rangle_\mu$ if and only if
\vspb
$$-(b\circ a, c)=(a, b\star c),\;\; (a\star b, c)=-(a, b\circ c),\;\; (a, [b,c])=([a,b],c).$$
    Note that $(\cdot, \cdot)$ is symmetric. Therefore the conclusion holds.
\vspb
\end{proof}

\begin{defi}\mcite{L}
    A (finite) {\bf conformal Manin triple of Lie conformal algebras} is a triple of finite Lie conformal algebras $(R, R_0, R_1)$, where $R$ is equipped with a nondegenerate invariant symmetric conformal bilinear form $\langle \cdot, \cdot \rangle_\lambda$ such that
    \begin{enumerate}
        \item $R_0$ and $R_1$ are Lie conformal subalgebras of $R$, and $R=R_0\oplus R_1$ as ${\bf k}[\partial]$-modules;
        \item $R_0$ and $R_1$ are isotropic with respect to $\langle \cdot, \cdot \rangle_\lambda$, that is $\langle R_i, R_i \rangle_\lambda=0$ for $i=0$, $1$.
    \end{enumerate}
\end{defi}

\begin{pro}\mlabel{conf-Manin triple}
    Let $(A, \circ_A, [\cdot,\cdot]_A)$ and $(A^\ast, \circ_{A^\ast}, [\cdot,\cdot]_{A^\ast})$ be \gdas, and ${\bf k}[\partial]A$ and ${\bf k}[\partial]A^\ast$ be the corresponding Lie conformal algebras. Then $(
    {\bf k}[\partial](A\oplus A^\ast), {\bf k}[\partial]A, {\bf k}[\partial]A^\ast)$ is a conformal Manin triple with the bilinear form
\vspb
\begin{eqnarray}\mlabel{Manin triple1}
\langle a+f, b+g\rangle_\lambda = \langle f,
b\rangle+\langle g, a\rangle,\;\; a, b\in A, f, g\in
A^\ast,
\end{eqnarray}
if and only if $(A\oplus A^\ast, (A, \circ_A, [\cdot,\cdot]_A), (A^\ast, \circ_{A^\ast}, [\cdot,\cdot]_{A^\ast}))$ is a Manin triple of \gdas associated with the bilinear form defined by Eq. \meqref{bilinear}.
\end{pro}
\begin{proof}
    It follows from Proposition \mref{corresp-bilinear}.
\vspb
\end{proof}

\begin{pro}\mlabel{equi-bialgebra} Let $R$ be a finite Lie conformal algebra which is free as a ${\bf k}[\partial]$-module and $(R, \delta)$ be a Lie conformal coalgebra. Suppose that there is a Lie conformal algebra structure on $R^{\ast c}$ which is obtained from $\delta$. Then the following conditions are equivalent.
    \begin{enumerate}
        \item \mlabel{it:11}$(R, [\cdot_\lambda \cdot], \delta)$ is a Lie conformal bialgebra;
        \item \mlabel{it:12} $(R\oplus R^{\ast c}, R, R^{\ast c})$ is a conformal Manin triple of Lie conformal algebras with the conformal bilinear form given by
\vspb
        \begin{eqnarray}
            \langle a+f, b+g\rangle_\lambda=f_\lambda(b)+g_{-\lambda}(a), \;\;\; a, b\in R, f, g \in R^{\ast c}.
        \end{eqnarray}
        \item \mlabel{it:13} $(R, R^{\ast c}, \ad_{R}^\ast, \ad_{R^{\ast c}}^\ast)$ is a matched pair of Lie conformal algebras.
    \end{enumerate}
\end{pro}

\begin{proof}
    By \cite[Theorem 3.9]{L}, Conditions~(\mref{it:11}) and (\mref{it:12})  are equivalent. With a similar proof as that in \cite[Theorem 3.6]{HB2}, one can check that Conditions~(\mref{it:12}) and (\mref{it:13}) are equivalent.
\vspb
\end{proof}

\begin{pro}
    Let $(A, \circ_A, [\cdot,\cdot]_A, \Delta, \delta_0)$ be a \gdba and $R={\bf k}[\partial]A$ be the Lie conformal algebra corresponding to  $(A, \circ_A, [\cdot,\cdot]_A)$. Then
    the Lie conformal bialgebra $(R, [\cdot_\lambda \cdot], \delta)$ with $\delta$ defined by Eq.~(\mref{cobracket}) is exactly the one obtained from the conformal Manin tripe $({\bf k}[\partial](A\oplus A^\ast)$, ${\bf k}[\partial]A$, ${\bf k}[\partial]A^\ast)$ given in Proposition \mref{conf-Manin triple} or the matched pair $(R, R^{\ast c}, \ad_R^\ast, \ad_{R^{\ast c}}^\ast)$ given in Corollary \mref{equi-matched pair}.
\end{pro}
\begin{proof}
    Since $(A, \circ_A, [\cdot,\cdot]_A, \Delta, \delta_0)$ is a \gdba, the triple  $(A\oplus A^\ast, (A, \circ_A, [\cdot,\cdot]_A), (A^\ast, \circ_{A^\ast},$ $ [\cdot,\cdot]_{A^\ast}))$ is a Manin tripe of \gdas associated with the bilinear form given by Eq. (\mref{bilinear}) and the octuple $(A, A^\ast, L_{\star_A}^\ast, $ $-R_{\circ_A}^\ast, $ $\ad_A^\ast, L_{\star_{A^\ast}}^\ast, -R_{\circ_{A^\ast}}^\ast, \ad_{A^\ast}^\ast)$ is a matched pair of \gdas by Corollary \mref{GD-equi}. By Proposition \mref{conf-Manin triple} and Corollary \mref{equi-matched pair}, $({\bf k}[\partial](A\oplus A^\ast), {\bf k}[\partial]A, {\bf k}[\partial]A^\ast)$ is a conformal Manin triple and $(R, R^{\ast c}, \ad_R^\ast, \ad_{R^{\ast c}}^\ast)$ is a matched pair of Lie conformal algebras. By Proposition \mref{equi-bialgebra}, the Lie conformal bialgebra $(R, [\cdot_\lambda \cdot], \overline{\delta})$ obtained from the conformal Manin tripe $({\bf k}[\partial](A\oplus A^\ast), {\bf k}[\partial]A, {\bf k}[\partial]A^\ast)$ or  the matched pair $(R, R^{\ast c}, \ad_R^\ast, \ad_{R^{\ast c}}^\ast)$ is given by
\vspb
\begin{eqnarray*}
\langle[f_\mu g], a\rangle_\lambda =\sum_{(a)}f_\mu(a_{(1)})g_{\lambda-\mu}(a_{(2)})=(f\otimes g)_{\mu, \lambda-\mu}(\overline{\delta}(a)), \;\; a\in A, f, g\in A^\ast,
\vspb
\end{eqnarray*}
where $\overline{\delta}(a)=\sum_{(a)}a_{(1)}\otimes a_{(2)}$. Let $a\in A, f, g\in A^\ast$. Set $\delta_0(a)=\sum a_{(1)}\otimes a_{(2)}$ and $\Delta(a)=\sum \overline{a_{(1)}}\otimes \overline{a_{(2)}}$.
Then we have
\vspb
\begin{eqnarray*}
(f\otimes g)_{\mu, \lambda-\mu}(\overline{\delta}(a))&=&\langle [f_\mu g], a\rangle_\lambda\\
&=&\langle \partial (g\circ_{A^\ast} f)+\mu(f\circ_{A^\ast} g+g\circ_{A^\ast} f)+[f,g]_{A^\ast}, a\rangle_\lambda\\
&=&-\lambda \langle g\circ_{A^\ast} f, a \rangle+\mu \langle f\circ_{A^\ast} g+g\circ_{A^\ast} f, a \rangle+\langle [f,g]_{A^\ast}, a \rangle\\
&=& -\lambda \langle g\otimes f, \Delta(a) \rangle +\mu \langle f\otimes g+g\otimes f, \Delta(a) \rangle+\langle f\otimes g, \delta_0(a)\rangle,\\
(f\otimes g)_{\mu, \lambda-\mu}(\delta(a))&=&(f\otimes g)_{\mu, \lambda-\mu}((\partial\otimes \id)\Delta(a)-\tau(\partial\otimes \id)\Delta(a)+\delta_0(a))\\
&=& (f\otimes g)_{\mu, \lambda-\mu}(\sum(\partial \overline{a_{(1)}}\otimes \overline{a_{(2)}}-\overline{a_{(2)}}\otimes \partial \overline{a_{(1)}})+\sum a_{(1)}\otimes a_{(2)})\\
&=& \sum(f_\mu (\partial \overline{a_{(1)}})g_{\lambda-\mu}(\overline{a_{(2)}})-f_\mu(\overline{a_{(2)}})g_{\lambda-\mu}(\partial \overline{a_{(1)}}))+\sum f(a_{(1)})g(a_{(2)})\\
&=&\mu \sum f(\overline{a_{(1)}})g(\overline{a_{(2)}})-(\lambda-\mu) \sum f(\overline{a_{(2)}})g(\overline{a_{(1)}})+\sum f(a_{(1)})g(a_{(2)})\\
&=& -\lambda \langle g\otimes f, \Delta(a) \rangle +\mu \langle f\otimes g+g\otimes f, \Delta(a) \rangle+\langle f\otimes g, \delta_0(a)\rangle.
    \end{eqnarray*}
    Hence $(f\otimes g)_{\mu, \lambda-\mu}(\overline{\delta}(a))=(f\otimes g)_{\mu, \lambda-\mu}(\delta(a))$. Therefore
    $\overline{\delta}(a)=\delta(a)$. Then the proof is completed.
\end{proof}

By the above discussion, we have the following commutative diagram.
\begin{displaymath}
    \xymatrix{
        \txt{\tiny $(A, A^\ast, L_{\star_A}^\ast, -R_{\circ_A}^\ast, \ad_A^\ast, L_{\star_{A^\ast}}^\ast, -R_{\circ_{A^\ast}}^\ast, \ad_{A^\ast}^\ast)$\\ \tiny a matched pair of\\ \tiny \gdas}\ar@{<->}[rr]^-{\rm Cor. \mref{GD-equi}}  \ar@{<->}[d]_{\rm Cor. \mref{equi-matched pair}}&&
        \txt{\tiny $(A, \circ_A, [\cdot,\cdot]_A, \Delta, \delta_0)$\\ \tiny a \gdba} \ar@{<->}[rr]^-{\rm Cor. \mref{GD-equi}} \ar@{<->}[d]^(.6){\rm Thm. \mref{consthm1}}&& \txt{\tiny $(A\oplus A^\ast, (A, \circ_A, [\cdot,\cdot]_A), (A^\ast, \circ_{A^\ast}, [\cdot,\cdot]_{A^\ast}))$\\ \tiny a Manin tripe of \\ \tiny \gdas} \ar@{<->}[d]_{\rm Prop. \mref{conf-Manin triple}}\\
        \txt{\tiny $(R={\bf k}[\partial]A, R^{\ast c}, \ad_R^\ast, \ad_{R^{\ast c}}^\ast)$\\ \tiny a matched pair of \\ \tiny Lie conformal algebras}\ar@{<->}[rr]^-{\rm Prop. \mref{equi-bialgebra}}  &&
        \txt{\tiny  $(R={\bf k}[\partial]A, [\cdot_\lambda \cdot], \delta)$\\ \tiny a Lie conformal bialgebra} \ar@{<->}[rr]^-{\rm Prop. \mref{equi-bialgebra}} && \txt{\tiny $({\bf k}[\partial](A\oplus A^\ast), {\bf k}[\partial]A, {\bf k}[\partial]A^\ast)$ \\ \tiny a conformal Manin triple\\  \tiny of Lie conformal algebras}
    }
\end{displaymath}

\noindent {\bf Acknowledgments.} This research is supported by
NSFC (12171129, 11931009, 12271265, 12261131498, 12326319),
 the Fundamental Research Funds for the Central Universities and Nankai Zhide Foundation.

\smallskip

\noindent
{\bf Declaration of interests. } The authors have no conflicts of interest to disclose.

\smallskip

\noindent
{\bf Data availability. } No new data were created or analyzed in this study.

\end{document}